\pdfoutput=1
\RequirePackage{ifpdf}
\ifpdf 
\documentclass[pdftex]{sigma}
\else
\documentclass{sigma}
\fi

\usepackage{mathrsfs,stmaryrd,mathtools}
\usepackage{enumerate}

\usepackage{tikz}
\usetikzlibrary{arrows,calc,decorations.pathreplacing,decorations.markings,intersections,shapes.geometric,through,fit,shapes.symbols,positioning,decorations.pathmorphing}

\numberwithin{equation}{section}

\newtheorem{Theorem}{Theorem}[section]
\newtheorem*{Theorem*}{Theorem}
\newtheorem{Corollary}[Theorem]{Corollary}
\newtheorem{Lemma}[Theorem]{Lemma}
\newtheorem{Proposition}[Theorem]{Proposition}
 { \theoremstyle{definition}
\newtheorem{Definition}[Theorem]{Definition}

\newtheorem{Notation}[Theorem]{Notation}
\newtheorem{Example}[Theorem]{Example}
\newtheorem{Examples}[Theorem]{Examples}
\newtheorem{Remark}[Theorem]{Remark}
\newtheorem{Remarks}[Theorem]{Remarks}
\newtheorem{Convention}[Theorem]{Convention}
}

\newcommand\bC{{\mathbb C}}

\newcommand\bH{{\mathbb H}}

\newcommand\bP{{\mathbb P}}

\newcommand\bR{{\mathbb R}}
\newcommand\bS{{\mathbb S}}

\newcommand\bZ{{\mathbb Z}}

\newcommand\cE{{\mathcal E}}

\newcommand\cU{{\mathcal U}}
\newcommand\cV{{\mathcal V}}

\newcommand\fg{{\mathfrak g}}

\newcommand\fk{{\mathfrak k}}

\newcommand\fu{{\mathfrak u}}

\newcommand{\cat}[1]{\text{\sc #1}}

\DeclareMathOperator{\id}{id}

\DeclareMathOperator{\Hom}{\mathrm{Hom}}
\DeclareMathOperator{\im}{\mathrm{im}}

\begin{document}
\allowdisplaybreaks

\newcommand{\arXivNumber}{2211.11429}

\renewcommand{\PaperNumber}{106}

\FirstPageHeading

\ShortArticleName{Manifolds of Lie-Group-Valued Cocycles and Discrete Cohomology}

\ArticleName{Manifolds of Lie-Group-Valued Cocycles\\ and Discrete Cohomology}

\Author{Alexandru CHIRVASITU and Jun PENG}

\AuthorNameForHeading{A.~Chirvasitu and J.~Peng}

\Address{Department of Mathematics, University at Buffalo, Buffalo, NY 14260-2900, USA}
\Email{\href{mailto:achirvas@buffalo.edu}{achirvas@buffalo.edu}, \href{mailto:jpeng3@buffalo.edu}{jpeng3@buffalo.edu}}

\ArticleDates{Received June 18, 2023, in final form December 01, 2023; Published online December 24, 2023}

\Abstract{Consider a compact group $G$ acting on a real or complex Banach Lie group~$U$, by automorphisms in the relevant category, and leaving a central subgroup $K\le U$ invariant. We define the spaces ${}_KZ^n(G,U)$ of $K$-relative continuous cocycles as those maps~${G^n\to U}$ whose coboundary is a $K$-valued $(n+1)$-cocycle; this applies to possibly non-abelian $U$, in which case $n=1$. We show that the ${}_KZ^n(G,U)$ are analytic submanifolds of the spaces~$C(G^n,U)$ of continuous maps $G^n\to U$ and that they decompose as disjoint unions of fiber bundles over manifolds of $K$-valued cocycles. Applications include: (a) the fact that~${Z^n(G,U)\subset C(G^n,U)}$ is an analytic submanifold and its orbits under the adjoint of the group of $U$-valued $(n-1)$-cochains are open; (b) hence the cohomology spaces $H^n(G,U)$ are discrete; (c) for unital $C^*$-algebras $A$ and $B$ with $A$ finite-dimensional the space of morphisms $A\to B$ is an analytic manifold and nearby morphisms are conjugate under the unitary group~$U(B)$; (d) the same goes for $A$ and $B$ Banach, with $A$ finite-dimensional and semisimple; (e) and for spaces of projective representations of compact groups in arbitrary~$C^*$ algebras (the last recovering a result of Martin's).}


\Classification{22E65; 17B65; 58B25; 22E41; 57N35; 46L05; 16H05; 16D60; 16K20}

\section{Introduction}

The results below all revolve around the familiar pattern of rigidity in the presence of semisimplicity. Both terms are ambiguous and multiform, but roughly speaking:
\begin{itemize}\itemsep=0pt
\item `rigidity' means that ``nearby'' structures (morphisms, derivations, cocycles, etc.) are mutually equivalent, whatever that might mean;
\item while `semisimplicity' is a stand-in for any of a number of behaviors having to do with cohomology vanishing of some type: semisimplicity in the representation-theoretic sense~\cite[Definition~2.1]{lam} for Lie or associative algebras or for compact groups (which amounts to the vanishing of $\mathrm{Ext}$ functors), {\it separability} \cite[Section~10.2, Definition]{pierce} for associative algebras (i.e., Hochschild-cohomology triviality \cite[Section~11.5, Proposition]{pierce}), etc.
\end{itemize}

By way of an illustration, and to give as good an entry point to the discussion as any other, consider \cite[Proposition 5.2.6]{wo}: given a projection $p$ of a unital $C^*$-algebra~$B$, all projections~$q$ sufficiently close to it can be conjugated back onto $p$ by means of unitary elements $u\in B$ chosen continuously in $q$. This is the type of rigidity result alluded to above: projections in~$B$ are unital $C^*$-algebra morphisms $\bC^2\to B$ (see Remark \ref{re:wo}), and the recorded close-hence-conjugate phenomenon turns out to be a consequence of the semisimplicity of $\bC^2$ (or rather its separability).

The natural next step is to ask whether the analogue holds for morphisms $A\to B$ for arbitrary finite-dimensional $C^*$-algebras $A$. Up to minor adjustments the reader can easily supplement (see Remark \ref{res:wogen}\,(\ref{item:acs})), this is essentially the content of \cite[Proposition~2.2]{acs}. The many variations on the theme, in differing contexts but bearing the unmistakable family resemblance, suggest a~broader phenomenon. To wit:
\begin{enumerate}[(i)]\itemsep=0pt
\item\label{item:8} There are analogues for Banach (as opposed to $C^*$-)algebras (\cite[Proposition 8.2.2 and Theorem 8.2.3]{run-lec}, citing \cite[Proposition~4.1 and Theorem~4.5]{cg-amen}).
 If $A$ is an {\it amenable} Banach algebra (\cite[Definition 2.1.9]{run-lec}: a cohomology-vanishing condition) and $B$ is sufficiently well-behaved, the subspace
 \begin{equation*}
 \cat{BAlg}(A,B)\subset \cat{Ban}(A,B):=\text{all Banach-space morphisms}
 \end{equation*}
 is an analytic submanifold, and the action of the multiplicative group $B^{\times}$ of invertible elements splits locally. The latter condition means that, as before, sufficiently close morphisms are mutual conjugates.

\item\label{item:9} Virtually all of goes through in the slightly different setup of {\it projective representations} of compact groups \cite[Propositions~1.7, 2.2 and~2.3]{mart-proj}.
 Given a 2-cocycle $\sigma$ valued in $\bS^1$ for a compact group $G$ (compactness, again, being an avatar of semisimplicity), the space of projective $\sigma$-representations \cite[Section~1.4]{mart-proj} $G\to U(B)$ into the unitary group of a~$C^*$-algebra $B$ is an analytic submanifold of
 \begin{equation*}
 C(G,B):=\text{continuous functions }G\to B,
 \end{equation*}
 sufficiently close representations are mutual unitary conjugates, etc.

\item\label{item:10} The (possibly non-abelian) cohomology of a finite group valued in a compact Lie group is finite: see, e.g., \cite[Part~II, Chapter~IV, Exercise~1]{ser-lag}, which suggests using the manifold structure on the space of cocycles, or \cite[Theorems~1.1 and~1.2 and Corollary~1.1]{aw-nonab}, which proceed essentially in this manner.
 So once more: cocycles form a manifold, and sufficiently close cocycles are equivalent (here meaning cohomologous).
\end{enumerate}
The selection presumably suffices to confirm that the theme recurs.

To aggregate, summarize and paraphrase some of the results, we employ some of the standard notation familiar from the literature on (group, say) cohomology: $X^{n}(G,U)$ denotes $U$-valued $n$-cochains, cocycles, coboundaries and cohomology, depending on whether $X = C,Z,B$ or $H$ respectively; see Section~\ref{subse:coh} for a quick reminder.

The general setup entails fixing a short exact sequence $\cE$ of Banach Lie groups (real or complex) as in~\eqref{eq:gpext}, with $K\le U$ central and a Lie subgroup in the strong sense of \cite[Section~III.1.3, Definition~3]{bourb-lie-13}; closed, and such that the tangent subspaces to $K$ are {\it complemented} Banach subspaces of the tangent spaces to $U$.

Additionally, we have a compact group $G$ (not necessarily Lie), acting on all of $\cE$ in the sense that it leaves $K$ invariant. There will be a cocycle/cohomological degree $n$ in use at all times, with $n=1$ if $U$ is non-abelian.

In Definition \ref{def:relcocyc} and Notation \ref{not:relsigma} we introduce the spaces ${}_{\cE}Z^n(G,U)$ of {\it $\cE$-relative} cocycles, i.e., those continuous maps $G^n\to U$ whose coboundaries are $K$-valued $(n+1)$-cocycles. Similarly,
\begin{itemize}\itemsep=0pt
\item for a specific cocycle $\sigma\in Z^{n+1}(G,K)$, the space ${}_{\sigma}Z^n(G,U)$ consists of those continuous maps $G^n\to U$ whose coboundary is precisely $\sigma$;
\item and more generally, for a set $S\subset Z^{n+1}(G,K)$ of cocycles, ${}_{S}Z^n(G,U)$ consists of those maps $G^n\to U$ whose coboundary lies in $S$.
\end{itemize}

A selection of statements spread over Theorems \ref{th:cocyle-man-1}, \ref{th:cocyle-man-n}, \ref{th:relcocfiber}, \ref{th:relcocfiber-tg} and \ref{th:canconjrel}:

\begin{Theorem*}
 In the setting outlined above, the following statements hold.
 \begin{enumerate}[$(a)$]\itemsep=0pt

 \item The space ${}_{\cE}Z^n(G,U)$ of relative cocycles is a closed submanifold of the space $C(G^n,U)$ of continuous maps $G^n\to U$, decomposing as a disjoint union of open subsets ${}_S Z^n(G,U)$ for~$S$ ranging over the cohomology classes in $Z^n(G,K)$.

 \item\label{item:12} For every cohomology class $S\subset Z^n(G,K)$, if non-empty, the space ${}_S Z^{n}(G,U)$ fibers over~$S$. That fibration is associated to the principal $Z^n(G,K)$-bundle
 \begin{align*}
 C^n(G,K) \xrightarrow{\quad}
 C^n(G,K)/Z^n(G,K) \xrightarrow[\quad\cong\quad]{}
 B^{n+1}(G,K) &\cong B^{n+1}(G,K)\cdot \sigma=S
 \end{align*}
 $($for any fixed $\sigma\in S)$, resulting from the $Z^n(G,K)$-action on ${}_SZ^n(G,U)$ by multiplication.

 \item For a cocycle $\sigma\in Z^{n+1}(G,K)$ and $u\in {}_{\sigma}Z^n(G,U)$ we have
 \begin{equation*}
 T_u \left({}_{\sigma}Z^n(G,U)\right)
 \cong
 Z^n(G,\fu),
 \end{equation*}
 the space of cocycles valued in the Lie algebra $\fu:=\operatorname{Lie}(U)$ where, if $U$ is non-abelian and hence $n=1$, the action of $G$ on $\fu$ is the original one twisted by the cocycle $u$.

 \item The Lie group $C^{n-1}(G,U)\times C^n(G,K)$ admits action on ${}_S Z^n(G,U)$, whereby the first factor operates by a higher analogue of conjugation $($see Notation $\ref{not:actpart})$ and the second factor by multiplication.

 That action has local cross sections, so in particular its orbits are open.

 \item\label{item:13} The same goes for the action of $C^{n-1}(G,U)$ on single fibers ${}_{\sigma}Z^n(G,U)$, where $\sigma\!\!\in\!\! Z^{n+1}(G,K)$ is a cocycle.
 \end{enumerate}
\end{Theorem*}

The various results listed above as motivation can be recovered in this framework:
\begin{itemize}\itemsep=0pt
\item Manifolds of morphisms $A\to B$ of operator algebras, whether $C^*$ or Banach, with $A$ finite-dimensional and, in the Banach case, semisimple, are treated in Theorem \ref{th:tgsp}.

 This is a version of item (\ref{item:8}) above, except that $A$ is finite-dimensional (rather than amenable), while $B$ is arbitrary (rather than {\it dual} \cite[Theorem~8.2.3]{run-lec}).

\item Example \ref{exs:motivk}\,(\ref{item:motivk-martin}) and Remark \ref{re:martinsame} explain how the results of \cite{mart-proj} (mentioned in (\ref{item:9})) can be recast as instances of the present discussion: the manifolds $R(G,\sigma,A)$ of projective representations of a compact group $G$ in a $C^*$-algebra $A$, discussed in \cite[Section~1.4]{mart-proj}, are the fibers ${}_{\sigma}Z^1(G,U)$ of the fibration in part (\ref{item:12}) of the statement given above.

 In other words, the theorem stated here aggregates the various spaces $R(G,\sigma,A)$ studied in \cite{mart-proj} into a single manifold.

\item Finally, how (\ref{item:10}) relates to the contents of the theorem should be clear: for trivial $K\le U$, the statements revert back to plain (as opposed to relative) cocycles, the orbits mentioned in part (\ref{item:13}) are precisely the cohomology classes, etc.
\end{itemize}

\section{Preliminaries}\label{se.prel}

Apart from explicit warnings to the contrary, Banach algebras are assumed unital, as are morphisms. Per standard practice (e.g., \cite[Definition~2.1.8]{dales}), we assume throughout that in a~unital Banach algebra the multiplicative unit has norm~1: $\|1\|=1$.

Some assorted notation:
\begin{itemize}\itemsep=0pt
\item $\cat{Ban}$ is the category of Banach spaces and continuous linear maps, and $\cat{Ban}(E,F)$ or $L(E,F)$ is the space of continuous linear morphisms between two Banach spaces.

 Banach spaces (and algebras) are either real or complex: the discussion can be carried out in parallel, so the discussion can be kept agnostic of the ground field.

\item $\cat{BAlg}$ is the category of (unital) Banach algebras and continuous unital algebra morphisms (all of them, not just those of norm $\le 1$; see Remark \ref{re:notcontr}).

 The space of morphisms $A\to B$ is denoted by $\cat{BAlg}(A,B)$ is the space of unital-Banach-algebra morphisms $A\to B$: continuous, linear, unital, multiplicative maps.
\item Similarly, $C^*(A,B)$ is the space of unital $C^*$-algebra morphisms between unital $C^*$-al\-ge\-bras $A$ and $B$.
\item $A^{\times}$ is the multiplicative group of invertible elements in a Banach algebra $A$.
\item $U(A)$ is the unitary group of a $C^*$-algebra $A$.
\end{itemize}

\begin{Remark}\label{re:notcontr}
 While morphisms of $C^*$-algebras automatically have norm $1$ (as follows, for instance, from \cite[Proposition 1.6.2]{arv}), this is not the case, in general, for Banach algebras: the algebra $\bC[\varepsilon]/\big(\varepsilon^2\big)$ of {\it dual numbers}, equipped with an arbitrary Banach-algebra norm. Its endomorphisms are in bijection with $\bC$, determined uniquely by how they scale $\varepsilon$.

 It is not uncommon, in the literature, to work with the category $\cat{Ban}_1$ of Banach spaces and {\it short} linear maps, i.e., those of norm at most $1$: see \cite[Example 1.48]{ar}, \cite{poth}, etc. In many ways, $\cat{Ban}_1$ is better behaved than the larger category $\cat{Ban}$ of Banach spaces and arbitrary linear continuous maps. In that context one would then consider only Banach-algebra morphisms of norm $\le 1$; for our purposes though, this constraint is less natural and not particularly useful. For that reason, we emphasize that $\cat{BAlg}(A,B)$ consists of {\it all} continuous (complex, unital, algebra) morphisms, regardless of norm.
\end{Remark}

\subsection{Banach manifolds and Lie groups}\label{subse:banmf}

We need some background on (real) infinite-dimensional $C^p$ or smooth manifolds, modeled on Banach spaces, and the resulting notion of (Banach--)Lie group. There is no shortage of good sources in the literature, but they occasionally differ on the precise setup (e.g., in considering smooth vs.\ {\it analytic} manifolds).

\cite{bourb-vars-17,bourb-vars-8-15} give a very convenient uniform treatment of $C^p$ manifolds, where $p$ can be
\begin{itemize}\itemsep=0pt
\item a non-negative integer, whereby the $C^p$ condition (on a map, say) is defined recursively, as being differentiable with derivatives of class $C^{p-1}$ (and with $C^0$ signifying continuity);

\item the symbol `$\infty$', meaning of class $C^p$ for all non-negative integers $p$ (the term `smooth' is also customary for such manifolds/maps);

\item the symbol `$\omega$', meaning `analytic' (real, or complex, or over an ultrametric field).
\end{itemize}
As noted in \cite[Avertissement]{bourb-vars-17}, however, the work constitutes a summary of results and does not contain proofs. Lie {\it groups} are studied in a much more detailed fashion in \cite[Chapter III]{bourb-lie-13}, with the caveat that the text refers back to \cite{bourb-vars-17,bourb-vars-8-15} frequently.

Banach manifolds also receive a thorough treatment in \cite{lang-fund}, which focuses on the {\it smooth} (rather than analytic) setting. In particular, the Lie groups of \cite[Section~VI.5]{lang-fund} are required to be only ($C^p$ or) smooth; the same goes for those of \cite[Definition IV.1]{neeb-inf}. One the other hand, the Banach manifolds \cite[Section~3]{upm-ban} and Lie groups of \cite[Section~6]{upm-ban} are analytic.

Below, we work mostly with analytic (rather than smooth) manifolds and especially Lie groups. The ground field can be either $\bR$ or $\bC$, and the discussion will go through unaffected regardless of the choice. Over $\bC$ the analytic (i.e., holomorphic) setting is the only reasonable choice, while for real Banach {\it Lie groups} are automatically analytic \cite[Corollary IV.1.10]{neeb-lc}.

Natural examples of Lie groups include
\begin{itemize}\itemsep=0pt
\item Banach spaces with their additive structure \cite[Section~III.1.1, Examples following Proposition 3]{bourb-lie-13}, acting also as their own Lie algebras;
\item The multiplicative group $A^{\times}$ of invertible elements in a unital Banach algebra (\cite[Proposition IV.9]{neeb-inf} or \cite[Example 6.9]{upm-ban}, where the group is denoted by $G(A)$). The Lie algebra is $A$;
\item The unitary group $U(A)$ of a $C^*$-algebra $A$, with the space
 \begin{equation*}
 \{a\in A\mid a+a^*=0\}\subset A
 \end{equation*}
 of {\it skew-adjoint} elements as its Lie algebra \cite[Corollary~15.22]{upm-ban},
\end{itemize}
and so on.

\begin{Convention}\label{cv:qlie}
 Submanifolds $N\subseteq M$ and Lie subgroups (which will typically be closed), unless specified otherwise, are as in \cite[Section~5.8.3]{bourb-vars-17} (and \cite[Section~III.1.3, Definition~3]{bourb-lie-13}, \cite[Section~II.2]{lang-fund}, etc.): the requirement is that the resulting embeddings
 \begin{equation*}
 T_pN\le T_pM,\qquad p\in N
 \end{equation*}
 {\it split} as Banach-space morphisms, i.e., that there be closed subspaces
 \begin{equation*}
 X\le T_pM,\qquad T_pM\cong T_pN\oplus X.
 \end{equation*}

 By contrast, absent the splitting, the respective terms employed in \cite[Section~III.1.3, Remark]{bourb-lie-13} are {\it quasi-submanifold} and {\it Lie quasi-subgroup}. We use these too, on occasion.
\end{Convention}

\begin{Remark}
 These conventions are {\it not} universal: \cite[Definition IV.6]{neeb-inf}, for instance, requires closure of Lie subgroups, but not splitting. So,
 \begin{itemize}\itemsep=0pt
 \item the Lie quasi-subgroups of \cite[Section~III.1.3, Remark]{bourb-lie-13} are the (plain) Lie subgroups of \cite[Definition IV.6]{neeb-inf};
 \item while the Lie subgroups of \cite[Section~III.1.3, Definition 3]{bourb-lie-13} are termed {\it complemented} Lie subgroups in \cite[Definition IV.6]{neeb-inf}.
 \end{itemize}
\end{Remark}

We always assume Lie algebras $\fu:=\operatorname{Lie}(U)$ of Banach Lie groups to be equipped with norms making them into {\it Banach Lie algebras} \cite[discussion following Definition IV.1]{neeb-inf}: $(\fu,\|\cdot\|)$ is a~Banach space, and
\begin{equation*}
 \|[x,y]\|\le C\|x\| \|y\|,\qquad \forall x,y\in \fu \quad\text{for some}\quad C>0.
\end{equation*}

\section{Cocycle manifolds and discrete cohomology}\label{subse:cocyman}

In the discussion below, we will frequently have to work with spaces
\begin{equation*}
 C(X,M)
 :=
 \text{continuous maps }X\to M
\end{equation*}
for compact (always Hausdorff) spaces $X$ and Banach manifolds $M$. They will always be equipped with the {\it uniform topology} (or {\it compact-open topology} \cite[Definition preceding Theorem 46.8]{munk}, since $X$ is compact). For Banach {\it spaces} $E$, this topology on $C(X,E)$ is the one induced by the supremum (complete) norm
\begin{equation*}
 \|f\|:=\sup_{x\in X}\|f(x)\| = \max_{x\in X}\|f(x)\|,\qquad \forall f\in C(X,E).
\end{equation*}

As explained in \cite[Section~3.2]{ps-loop}, for a Banach Lie group $U$ with Lie algebra $\fu$ the space $C(X,U)$ is again a Banach Lie group, with $C(X,\fu)$ as its Lie algebra.

\subsection{Group cohomology (possibly non-abelian)}\label{subse:coh}

The setup will be that of compact (occasionally finite) groups $G$ acting on Banach Lie groups~$U$ so as to preserve the relevant structure (i.e., by analytic group automorphisms). Actions are always assumed continuous in the strongest reasonable sense:
\begin{equation*}
 G\times U\to U
\end{equation*}
is continuous. One frequently considers weaker notions of continuity (e.g., \cite[Section~X.1]{tak2}), but those will not play a role below.

For finite $G$, the reader can find quick treatments of the pertinent material on the resulting
\begin{itemize}\itemsep=0pt
\item {\it cohomology groups} $H^n(G,U)$, $n\in \bZ_{\ge 0}$ for abelian $U$;
\item and $H^n(G,U)$, $n=0,1$ for arbitrary $U$
\end{itemize}
in, say, \cite[Chapter VII and its Appendix]{ser-lf} or \cite[Sections~I.2 and I.5]{ser-gal}.

Briefly, the non-abelian machinery is as follows. Given an action
\begin{equation}\label{eq:gactu}
 G\times U\ni (g,u)\xmapsto{\quad} {}^gu\in U
\end{equation}
of a group $G$ (typically finite, for us) on a Lie group $U$ (possibly non-commutative and frequently infinite-dimensional), the {\it $U$-valued cocycles} attached to the action are the maps
\begin{equation}\label{eq:gtoug}
 G\ni g\mapsto u_g\in U
\end{equation}
satisfying the cocycle condition
\begin{equation}\label{eq:cocyccond}
 u_{gh} = u_g\cdot {}^g u_h,\qquad \forall g,h\in G.
\end{equation}
The space of all such cocycles is denoted by $Z^1(G,U)$. Two cocycles $(u_g)_g$ and $(v_g)_g$ are declared {\it equivalent} or {\it cohomologous} if there is some $w\in U$ with
\begin{equation*}
 v_g = w^{-1}\cdot u_g\cdot {}^g w,\qquad \forall g\in G,
\end{equation*}
and the set $H^1(G,U)$ of equivalence classes is the {\it first cohomology set} of $G$ with values in $U$.

These objects appear often in the theory of operator algebras (e.g., \cite[Definition X.1.4]{tak2}), in the process of deforming or ``twisting'' group actions on said algebras.

A smattering of other symbols featuring below:
\begin{itemize}\itemsep=0pt
\item $C^n(G,U):=C(G^n,U)$ is simply the space of maps $G^n\to U$ (which maps we refer to as {\it $n$-cochains}). It admits an obvious analytic-manifold structure, topology, etc., and is a~group when $U$ is abelian.

 Suppose $U$ is abelian. As explained in \cite[Section~VII.3]{ser-lf}, the space $C^n(G,U)$ can also be identified with that of $G$-module maps $G^{n+1}\to U$ (with $G$ acting on the domain diagonally, by left multiplication).
\item We have differentials
\[C^n(G,U)\xrightarrow{\quad\delta^n\quad}C^{n+1}(G,U),\] which compute the {\it cohomology} $H^n(G,U)$ as $
 H^n(G,U) = \ker \delta^n/\im \delta^{n-1}$.

\item $Z^n(G,U)\subset C^n(G,U)$ is the space of {\it $n$-cocycles} (with $n=0,1$ if $U$ is non-abelian). For abelian $U$, this is simply $Z^n(G,U):=\ker \delta^n$.
\item When $U$ is abelian,
 \begin{equation*}
 B^n(G,U):=\im\delta^{n-1}\subseteq Z^n(G,U)
 \end{equation*}
 is the space of {\it $n$-coboundaries}. We thus have
 \begin{equation*}
 H^n(G,U) = Z^n(G,U)/B^n(G,U).
 \end{equation*}
\end{itemize}
In all of the above, maps $G^n\to U$ are assumed continuous when working with compact $G$.

\subsection{Main results}\label{subse:gpmain}

\begin{Lemma}\label{le:cansplit}
 Let $G$ be a compact group acting linearly and continuously on a Banach space $E$ via
 \begin{equation*}
 G\times E\ni (g,e)\mapsto g\triangleright e\in E.
 \end{equation*}
 For every $n\in \bZ_{\ge 0}$, the map
 \begin{equation*}
 C^n(G,E)\xrightarrow{\quad\delta^n\quad} Z^{n+1}(G,E)
 \end{equation*}
 is a split surjection. Furthermore, we can choose splittings of norm $\le \sup_{g\in G}\|g\triangleright\|$.
\end{Lemma}
\begin{proof}
 The procedure is well known, and hinges on the ability to average against the Haar probability measure $\mu$ of $G$. In a slightly different context, the construction is described in \cite[Proof of Theorem~2.3]{moo-12} (which in turn cites earlier sources): given an $(n+1)$-cocycle ${a\in Z^{n+1}(G,E)}$,
 \begin{equation}\label{eq:splitgpcoc}
 b(g_1,\dots,g_{n}):=\int_{G}g^{-1}\triangleright a(g,\ g_1,\dots,g_n)\, \mathrm{d}\mu(g)
 \end{equation}
 is an $n$-cochain whose coboundary is precisely $a$.
\end{proof}

In the statement of Theorem \ref{th:cocyle-man-1} we will refer to the process of {\it twisting} an action by a~cocycle (e.g., \cite[Section~I.5.3, Example 2 preceding Proposition 34]{ser-gal} or \cite[Definition X.1.4, equation~(8)]{tak2}): if \eqref{eq:gactu} is an action and
\begin{equation*}
 u = (u_g)_g\in Z^1(G,U)
\end{equation*}
a 1-cocycle for it, one can define a new action by
\begin{equation*}
 g\triangleright x := u_g \cdot {}^g x \cdot u_g^{-1},\qquad g\in G,\quad x\in U.
\end{equation*}
Theorem \ref{th:cocyle-man-1} refers to the Lie-algebra version of this instead; see also Lemma \ref{le:equivact} for a more elaborate discussion in a slightly extended setup.

\begin{Theorem}\label{th:cocyle-man-1}
 Let $G$ be a compact group acting on a Banach Lie group $U$ with Lie algebra $\fu$.
 \begin{enumerate}[$(1)$]\itemsep=0pt
 \item\label{item:cocyle-man-1} The space
 \begin{equation*}
 Z^1(G,U)\subset C(G,U)
 \end{equation*}
 of $1$-cocycles is a closed analytic submanifold in the sense of {\rm \cite[\emph{Section}~5.8.3]{bourb-vars-17}}.

 \item\label{item:cocyle-man-1-tgz1} The tangent space to that subvariety at a cocycle $u\in Z^1(G,U)$ is $\rho_{u} Z^1(G,\fu)$, where
 \begin{itemize}\itemsep=0pt
 \item $\rho_{u}$ indicates right translation of tangent vectors by ${u}\in C(G,U)$;
 \item and the action `$\triangleright$' of $G$ on $\fu$ is the original one, twisted by the cocycle ${u}$:
 \begin{equation}\label{eq:defact}
 g\triangleright X:=\operatorname{Ad}_{u_g} {}^g X,\qquad \forall g\in G,\quad X\in \fu,
 \end{equation}
 with `$\operatorname{Ad}$' denoting the adjoint action of $U$ on its Lie algebra.
 \end{itemize}

 \item\label{item:cocyle-man-1-h1disc} The cohomology set $H^1(G,U)$, equipped with its quotient topology, is discrete.

\item\label{item:cocyle-man-1-coboundsplit} Every cocycle $u=(u_g)_g$ has a neighborhood $\cU$ in $Z^1(G,U)$ admitting an analytic map
 \begin{equation*}
 \cU\ni \left({u'}=(u'_g)_g\right)\xmapsto{\quad} v_{u'}\in U
 \end{equation*}
 such that
 \begin{align*}
 v_{u} =1\qquad\text{and}\qquad
 u'_g =v_{u'}^{-1}\cdot u_g\cdot {}^gv_{u'},\qquad \forall g\in G,\quad u'\in \cU.
 \end{align*}
 \end{enumerate}
\end{Theorem}
\begin{proof}
 We address the claims sequentially.

 \emph{Parts {\rm (\ref{item:cocyle-man-1})} and {\rm(\ref{item:cocyle-man-1-tgz1})}:} Fix, throughout the discussion, a 1-cocycle
 \begin{equation*}
 u=(u_g)_{g\in G}\in Z^1(G,U)
 \end{equation*}
 and the resulting twisted action \eqref{eq:defact} of $G$ on the Lie algebra $\fu$.

 A sufficiently small neighborhood $\cV$ of $u$ in the Lie group $C(G,U)$ consists of elements of the form
 \begin{equation*}
 v=(v_g)_g,\qquad v_g = \exp(X_g)\cdot u_g
 \end{equation*}
 for small-norm elements $X_g\in \fu$, with $\exp\colon\fu\to U$ the {\it exponential} (or {\it exponential function} \cite[discussion following Definition IV.1]{neeb-inf} ) of the Lie group.

 The additional condition that $v$ itself be a 1-cocycle reads
 \begin{equation*}
 \exp(X_{gh})u_{gh} = \exp(X_g)u_g \cdot \exp\left({}^gX_h\right){}^gu_h.
 \end{equation*}
 Or:
 \begin{align*}
 \exp(X_{gh})u_{gh} &= \exp(X_g)\cdot u_g\exp({}^gX_h)u_g^{-1}\cdot u_g\cdot{}^gu_h\\
 &= \exp(X_g)\cdot u_g\exp({}^gX_h)u_g^{-1}\cdot u_{gh}
 \quad\text{by \eqref{eq:cocyccond}}\\
 &=\exp(X_g)\cdot \exp(\operatorname{Ad}_{u_g}\,{}^gX_h)\cdot u_{gh}
 \quad\text{by the definition of the adjoint action}\\
 &=\exp(X_g)\cdot \exp(g\triangleright X_h)\cdot u_{gh}
 \quad\text{by the definition of the action `$\triangleright$'}.
 \end{align*}
 The rightmost factor cancels out to give
 \begin{equation*}
 \exp(X_{gh}) = \exp(X_g)\cdot \exp(g\triangleright X_h).
 \end{equation*}
 By the {\it BCH} formula ({\it Baker--Campbell--Hausdorff}: \cite[Definition IV.1.3 and Corollary IV.1.10]{neeb-lc}), this translates to
 \begin{equation*}
 X_{gh} = X_g + g\triangleright X_h + \frac 12\left[X_g,g\triangleright X_h\right] + \cdots.
 \end{equation*}
 When all norms $\|g\triangleright X_h\|$ are sufficiently small (which we may as well assume), the right-hand partial sum starting with the summand $\frac 12[-,-]$ onward is $O(\|X_g\|\cdot \|g\triangleright X_h\|)$ in ``big O notation''~\cite{bigo}: its norm is dominated by $C\|X_g\|\cdot \|g\triangleright X_h\|$ for some fixed constant $C>0$.

 If $\|(X_g)_g\| = \max_{g\in G} \|X_g\|$ is of order $\varepsilon$,
 \begin{equation*}
 \left(X_{gh} - X_g - g\triangleright X_h\right)_{g,h}\in B^2(G,\fu)
 \end{equation*}
 is a 2-coboundary and hence a 2-cocycle of size $O(\varepsilon^2)$. It follows from Lemma \ref{le:cansplit} that we can recover it as the coboundary of a {\it smaller-norm} cochain:
 \begin{equation*}
 X_{gh} - X_g - g\triangleright X_h
 =
 Y_{gh} - Y_g - g\triangleright Y_h
 \end{equation*}
 with $\|g\triangleright Y_{\bullet}\|=O\big(\varepsilon^2\big)$, i.e., $(X_g-Y_g)_g\in Z^1(G,\fu)$.
 Furthermore, the same Lemma \ref{le:cansplit} also ensures that $(X_g)_g\mapsto (Y_g)_g$ can be chosen to be analytic.

 The fact that $(Y_g)_g$ is in $O\big(\varepsilon^2\big)$ whenever $(X_g)_g$ is of order $\varepsilon$ also ensures that this map can be inverted, so we have an analytic chart on $\cU:=\cV\cap Z^1(G,U)$:
 \begin{equation*}
 \cU\ni (\exp(X_g)\cdot u_g)_g
 \xmapsto{\quad}
 (X_g)_g
 \xmapsto{\quad}
 (X_g-Y_g)\in Z^1(G,\fu),
 \end{equation*}
 onto some neighborhood of $0\in Z^1(G,\fu)$. This suffices to prove (\ref{item:cocyle-man-1}) and (\ref{item:cocyle-man-1-tgz1}): the notion of subvariety in the statement requires that the tangent space
 \begin{equation*}
 T_u Z^1(G,U)\le T_u C(G,U)
 \end{equation*}
 split, which is indeed the case, once more by Lemma \ref{le:cansplit}.

\emph{ Parts {\rm(\ref{item:cocyle-man-1-h1disc})} and {\rm(\ref{item:cocyle-man-1-coboundsplit})}:} Fix $u=(u_g)_g\in Z^1(G,U)$.
 The differential of the map
 \begin{equation}\label{eq:noncommdiff}
 U\ni v\mapsto \big(v^{-1}\cdot u_g\cdot {}^g v\big)_g\in Z^1(G,U)
 \end{equation}
 at the origin $1\in U$ is nothing but the cochain differential $C^0(G,\fu)\to C^1(G,\fu)$ (for the twisted action `$\triangleright$' discussed above), so it surjects onto the tangent space
 \begin{equation*}
 T_u Z^1(G,U) = Z^1(G,\fu) = B^1(G,\fu)
 \end{equation*}
 (by part (\ref{item:cocyle-man-1-tgz1})) with split kernel (see Lemma \ref{le:cansplit}).

 This makes \eqref{eq:noncommdiff} a {\it submersion} in the sense of \cite[Section~5.9.1\,(i)]{bourb-vars-17}, so that it also satisfies the other, equivalent conditions of \cite[Section~5.9.1]{bourb-vars-17}:
 \begin{itemize}\itemsep=0pt
 \item By \cite[Section~5.9.1\,(iv)]{bourb-vars-17} some neighborhood of $1\in U$ maps, through \eqref{eq:noncommdiff}, {\it onto} some open neighborhood of $u\in Z^1(G,U)$. This means that cocycles sufficiently close to $u$ are cohomologous to it, hence the discreteness of the cohomology set. This proves part (\ref{item:cocyle-man-1-h1disc}).
 \item And, also by \cite[Section~5.9.1\,(iv)]{bourb-vars-17}, \eqref{eq:noncommdiff} admits a {\it local section}: precisely what (\ref{item:cocyle-man-1-coboundsplit}) requires.
 \end{itemize}
 This concludes the proof as a whole.
\end{proof}

\begin{Remark}\label{re:serrex}
 Something along the lines of the above argument is presumably what is intended in \cite[Part II, Chapter IV, Exercise 1]{ser-lag}, which asks the reader to show that for
 \begin{itemize}\itemsep=0pt
 \item a finite group $G$;
 \item acting on a {\it compact} (finite-dimensional) Lie group $U$ over a locally compact field (possibly ultrametric, possibly of positive characteristic);
 \item with $|G|$ coprime to the characteristic of $k$
 \end{itemize}
 the cohomology set $H^1(G,U)$ is finite.

 That same exercise also asks the same of the higher cohomology groups $H^n(G,U)$, $n\ge 1$ for abelian $U$; see Theorem \ref{th:cocyle-man-n} below for the version relevant to the present setting.
\end{Remark}

Note the following consequence of Theorem \ref{th:cocyle-man-1}:

\begin{Corollary}\label{cor:nearbymor}
 For a compact group $G$ and Banach Lie group $U$ the space of morphisms $G\to U$ is a submanifold of $C(G,U)$, and sufficiently close morphisms are conjugate in $U$.
\end{Corollary}
\begin{proof}
 This is what Theorem \ref{th:cocyle-man-1} specializes to when $G$ acts on $U$ {\it trivially}: cocycles are morphisms, and cohomologous cocycles (i.e., morphisms) are conjugate.
\end{proof}

\begin{Remarks}\label{res:mz}\quad
 \begin{enumerate}[(1)]\itemsep=0pt

 \item\label{item:remz} For a cognate of Corollary \ref{cor:nearbymor}, see \cite[Theorem 1 and Corollary]{mz-conj}: compact subgroups of a~{\it finite-dimensional} Lie group are conjugate if they are ``sufficiently close'' to one another, in a sense that can be made precise.

 Those results, in fact, can be used to give an alternative and much shorter argument for (most of) Theorem \ref{th:cocyle-man-1} in the finite-dimensional case.

 \item For a possibly infinite {\it discrete} group $G$ generated by $g_i$, $1\le i\le n$, it still makes sense to regard the space $\Hom(G,U)$ of group morphisms as a subset of
 \begin{equation*}
 U^n\cong \mathrm{functions} \{g_i\mid 1\le i\le n\}\to U.
 \end{equation*}
 If $G$ is infinite however, none of Corollary \ref{cor:nearbymor} survives:
 \begin{itemize}\itemsep=0pt
 \item The quotient space by the conjugation action by $U$ need not be discrete in general. Indeed, when $U$ is algebraic the quotients by that conjugation action (in the sense of {\it geometric invariant theory}: \cite[Chapter 1, Section~4]{fkm}) are the so-called {\it character varieties} of, say, \cite[Introduction]{ls} and the many references mentioned there. In other words, nearby morphisms $G\to U$ need not be conjugate.
 \item And moreover, $\Hom(G,U)\subset U^n$ will not be a submanifold: consider the case $G=\bZ^2$, in which case $\Hom(G,U)$ can be identified with the space of commuting elements in $U$.

 As \cite[paragraph following the proof of Theorem 3.1]{gs} notes, for instance, this space is not locally Euclidean for $U={\rm SO}(3)$ (see also \cite[p.~739]{baird} for $U={\rm SU}(2)$).
 \end{itemize}
 \end{enumerate}
\end{Remarks}

By essentially the same argument, Theorem \ref{th:cocyle-man-1} says somewhat more than Corollary \ref{cor:nearbymor}:

\begin{Corollary}\label{cor:morsplit}
 Let $G$ be a compact group and $U$ a Banach Lie group.

 The space
 \begin{equation*}
 \Hom(G,U)\subset C(G,U)
 \end{equation*}
 of continuous morphisms is a closed analytic submanifold, and any morphism $\varphi\colon G\to U$ has a~neighborhood $\cU\subset \Hom(G,U)$ admitting an analytic map $\cU\ni \psi\xmapsto{\quad}u_{\psi}\in U$ such that
 \begin{itemize}\itemsep=0pt
 \item $u_{\varphi}=1$;
 \item and for all $\psi\in \cU$ we have $\psi = u_{\psi}\cdot \varphi\cdot u_{\psi}^{-1}$. \qedhere
 \end{itemize}
\end{Corollary}

There is also an abelian-group version of Theorem \ref{th:cocyle-man-1}, concerning arbitrary rather than 1-cohomology. We omit the proof, as no essentially new ideas are needed: the only additional complications are notational.

\begin{Theorem}\label{th:cocyle-man-n}
 Let $G$ be a compact group acting on an abelian Banach Lie group $U$ with Lie algebra $\fu$ and $n\in \bZ_{>0}$.
 \begin{enumerate}[$(1)$]\itemsep=0pt
 \item\label{item:3} The space
 \begin{equation*}
 Z^n(G,U)\subset C(G^n,U)
 \end{equation*}
 of $n$-cocycles is a Lie subgroup in the sense of {\rm\cite[\emph{Section}~III.1.3, \emph{Definition}~3]{bourb-lie-13}}.
 \item\label{item:4} The tangent space to that subvariety at a cocycle ${u}\in Z^n(G,U)$ is $\rho_{u} Z^n(G,\fu)$, where $\rho_{u}$ indicates right translation of tangent vectors by ${u}\in C(G^n,U)$ and the action of $G$ on $\fu$ is the original one.
 In particular, the Lie algebra of $Z^n(G,U)$ is $Z^n(G,\fu)$.
 \item\label{item:5} The cohomology set $H^n(G,U)$, equipped with its quotient topology, is discrete.
 \item\label{item:6} Every $n$-cocycle ${u}$ has a neighborhood $\cU$ in $Z^n(G,U)$ admitting an analytic map
 \begin{equation*}
 \cU\ni {u'}\xmapsto{\quad} v_{u'}\in C^{n-1}(G,U)
 \end{equation*}
 such that
 \begin{align*}
 v_{u} =0\qquad\text{and}\qquad
 u' =u-\partial v_{u'},\qquad \forall u'\in \cU,
 \end{align*}
 where $\partial\colon C^{n-1}(G,U)\twoheadrightarrow B^n(G,U)$ is the usual differential. \qedhere
\end{enumerate}
\end{Theorem}

To expand on Theorem \ref{th:cocyle-man-1} somewhat, we will take a cue from \cite{mart-proj}. That paper studies {\it projective} unitary representations of compact groups (such as our $G$) in (unital) $C^*$-algebras. This means that
\begin{itemize}\itemsep=0pt
\item on the one hand the results therein (e.g., \cite[Proposition 1.7, 2.2 and 2.3]{mart-proj}) are somewhat narrower in scope, in that they only deal with maps from $G$ into the group of unitaries of a $C^*$-algebra;
\item while on the other hand, said results are broader in a different direction: they handle not just morphisms $G\to U(A)$ (which would fall within the ambit of Corollaries \ref{cor:nearbymor} and \ref{cor:morsplit}), but rather {\it projective} morphisms (or representations).

 Recall what this means (e.g., \cite[Section~1.4]{mart-proj}):

 Fix a {\it normalized} 2-cocycle $\sigma\in Z^2\big(G,\bS^1\big)$ for the trivial $G$-action on $\bS^1$, i.e., one that is trivial as soon as one of its arguments is (\cite[Section~III.1]{br-coh} or \cite[Section~2.3]{ev-coh}, for instance, use this same terminology).

 Then, one considers continuous unital maps $\varphi\colon G\to U(A)$ such that
 \begin{alignat*}{3}
 &\varphi(gh) = \sigma(g,h)\varphi(g)\varphi(h),\qquad &&\forall g,h\in G,&\\
 &\varphi(g)^*= \sigma\big(g,g^{-1}\big) \varphi\big(g^{-1}\big),\qquad &&\forall g\in G.&
 \end{alignat*}
\end{itemize}

There ought, then, to be a common generalization; we will see below that indeed there is.

\begin{Notation}\label{not:norm}
A `$\nu$' subscript will indicate normalization for cochains, cocycles, etc. There are thus spaces $C^n_{\nu}(G,U)$ and $Z^n_{\nu}(G,U)$ of normalized cochains and cocycles respectively for abelian~$U$,~$C^1_{\nu}(G,U)$ and $Z^1_{\nu}(G,U)$ for arbitrary $U$, etc.

Lemma \ref{le:cansplit} and Theorems \ref{th:cocyle-man-1} and \ref{th:cocyle-man-n} go through for normalized cochains/cocycles as well, as can be seen by modifying the proofs in the obvious fashion; we will not elaborate.
\end{Notation}

Throughout, fix a {\it short exact sequence}
\begin{equation}\label{eq:gpext}
 \cE \colon \
 \{1\}\to K\xrightarrow{\quad} U\xrightarrow{\quad} U/K\to \{1\}
\end{equation}
of Banach Lie groups, where
\begin{itemize}\itemsep=0pt
\item $K$ is in fact {\it central} in $U$;
\item as well as a Lie subgroup in the sense of \cite[Section~III.1.3, Definition~3]{bourb-lie-13} (in particular, closed in~$U$);
\item and $U/K$ is the resulting quotient, per \cite[Section~III.1.6, Proposition~11\,(ii)]{bourb-lie-13}.
\end{itemize}
Additionally, the (compact, as before) group $G$ acts on $U$, continuously and by automorphisms, preserving $\cE$ (i.e., leaving $K$ invariant as a set). Short exact sequences \eqref{eq:gpext} for which $K\le U$ is central will themselves be termed `central' (so that the phrase applies to both subgroups and sequences).

\begin{Examples}\label{exs:motivk}
 The two motivating instances the reader should keep in mind are:
 \begin{enumerate}[(1)]\itemsep=0pt
 \item that of trivial $K$, whereby the setup specializes to a $G$-action on $U$ and hence to the framework relevant to Theorem \ref{th:cocyle-man-1}.
 \item\label{item:motivk-martin} that of the embedding $(K\le U)
 :=
 \big(\bS^1\le U(A)\big)$ of the modulus-1 scalars into the unitary group of a $C^*$-algebra with $G$ acting trivially on everything in sight; this would be the setup pertinent to \cite{mart-proj}.
 \end{enumerate}
\end{Examples}

We now need some terminology.

\begin{Definition}\label{def:relcocyc}
 Consider a compact-group action on an exact sequence \eqref{eq:gpext} of Banach Lie groups.
 \begin{enumerate}[(1)]\itemsep=0pt
 \item An {\it $\cE$- $($or $K$-$)$relative} 1-cocycle is a continuous map \eqref{eq:gtoug} satisfying the cocycle condition~\eqref{eq:cocyccond} modulo $K$.

 In other words, $g\mapsto u_g$ becomes a 1-cocycle upon further composition with $U\to U/K$.
 \item The same terminology applies to higher cocycles provided $U$ is abelian.
 \end{enumerate}
 To indicate the relative nature of the cocycles, we decorate the relevant cocycle spaces with a~left-hand `$\cE$' (or $K$) subscript, with the two possible subscripts (`$\cE$' and `$\nu$', per Notation \ref{not:norm}) amenable to compounding:
 \begin{itemize}\itemsep=0pt
 \item ${}_{\cE}Z^1(G,U)$ is the space of $\cE$-relative 1-cocycles;
 \item ${}_{\cE}Z^1_{\nu}(G,U)$ is that of {\it normalized} $\cE$-relative 1-cocycles;
 \item ${}_{\cE}Z^n_{\nu}(G,U)$ that of normalized $\cE$-relative n-cocycles when $U$ is abelian, etc.
 \end{itemize}
\end{Definition}

\begin{Notation}\label{not:relsigma}
 In the context of Definition \ref{def:relcocyc}, consider an $\cE$-relative cocycle $u\in {}_{\cE}Z^n(G,U)$ (with $n=1$ if $U$ is not abelian). For abelian $U$ the coboundary $\partial u$ is an $(n+1)$-cocycle $ \partial u\in Z^{n+1}(G,K)$ valued in the subgroup $K\le U$. Similarly, for arbitrary $U$ and $n=1$, the ``coboundary''
 \begin{equation}\label{eq:nccobound}
 G^2\ni (g,h)\xmapsto{\quad} u_{gh}\cdot \left(u_g\cdot {}^g u_h\right)^{-1}\in K
 \end{equation}
 is a 2-cocycle in $Z^2(G,K)$.

 Conversely, given a cocycle $\sigma\in Z^{n+1}(G,K)$, we denote by
 \begin{equation*}
 {}_{\sigma}Z^{n}(G,U)\subset {}_{\cE}Z^n(G,U)
 \end{equation*}
 the space of those $\cE$-relative cocycles whose coboundary is $\sigma$.

 More generally, for a set $S\subseteq Z^{n+1}(G,K)$,
 \begin{equation*}
 {}_{S}Z^{n}(G,U)\subset {}_{\cE}Z^n(G,U)
 \end{equation*}
 is the set of relative $n$-cocycles whose coboundary belongs to $S$. In particular,
 \begin{itemize}\itemsep=0pt
 \item for $S=\{\sigma\}$ we recover ${}_{\sigma}Z^{n}(G,U)
 =
 {}_{\{\sigma\}}Z^{n}(G,U)$;
 \item and we will be particularly interested in cohomology classes $c\in H^{n+1}(G,K)$, pulling back to subsets $\pi^{-1}(c)$ \big(or $\pi^{-1}c$\big) via
 \begin{equation*}
 Z^{n+1}(G,K)\xrightarrow{\quad\pi\quad}H^{n+1}(G,K).
 \end{equation*}
 In that case, the notation specializes to ${}_{\pi^{-1}c}Z^{n}(G,U)$.
 \end{itemize}
\end{Notation}

\begin{Remark}\label{re:nccobound}
 Although coboundaries are problematic for cochains valued in non-abelian groups, in the present constrained context it does make sense to denote \eqref{eq:nccobound} by $\partial u$ and we will do so.
\end{Remark}

The subsequent discussion requires some background on fiber bundles, whether general or (typically right) principal over a Banach Lie group. The former are recalled briefly in \cite[Section~IV.3]{lang-fund}, and are also the {\it fibrations} of \cite[Section~6.1]{bourb-vars-17}. {\it Principal} fiber bundles on the other hand are treated in \cite[Section~6.2]{bourb-vars-17}; this material is also cited in \cite[Section~III.1]{bourb-lie-13}, to which we appeal below.

Recall, in particular, the following construct, discussed for instance in \cite[Section~6.5]{bourb-vars-17} (assume all spaces to carry analytic Banach manifold structures, so the context for fibrations is that of~\cite[Section~6.1.1]{bourb-vars-17}):

\begin{Definition}\label{def:assocbdle}
 Let $\pi\colon P\to B$ be a right principal $G$-bundle and $F$ a space equipped with a~left $G$-action. The {\it fiber bundle associated} to this data is the quotient
 \begin{equation*}
 P\times^G F:=P\times F/{\sim},
 \end{equation*}
 where the equivalence relation `$\sim$' is defined by
 \begin{equation*}
 (p,f)\sim \big(pg,g^{-1}f\big),\qquad \forall g\in G,\quad p\in P,\quad f\in F.
 \end{equation*}
 As explained in \cite[Section~6.5.2]{bourb-vars-17}, this is indeed a fiber bundle in the sense of \cite[Section~6.1.1]{bourb-vars-17}.
\end{Definition}

In Theorem \ref{th:relcocfiber}, we denote by $\pi$ the surjection of cocycles onto cohomology, as in Notation~\ref{not:relsigma}.

\begin{Theorem}\label{th:relcocfiber}
 Let $G$ be a compact group acting on a central exact sequence \eqref{eq:gpext} of Banach Lie groups, and $n$ a positive integer that is arbitrary if $U$ is abelian and $1$ otherwise.
 \begin{enumerate}[$(1)$]\itemsep=0pt
 \item\label{item:relcocfiber1} We have a partition
 \begin{equation*}
 {}_{\cE}Z^n(G,U) = \coprod_{c\in H^{n+1}(G,K)}{}_{\pi^{-1}c} Z^n(G,U)
 \end{equation*}
 into disjoint open sets.
 \item\label{item:relcocfiber2} Fix a class $c\in H^{n+1}(G,K)$ and an individual cocycle $\sigma\in \pi^{-1}c\subset Z^{n+1}(G,K)$. If its domain is non-empty, the map
 \begin{equation}\label{eq:relznfibers}
 {}_{\pi^{-1}c}Z^n(G,U)\xrightarrow{\quad\partial\quad}\pi^{-1}c
 \end{equation}
 is an analytic-manifold fibration with fiber ${}_{\sigma}Z^n(G,U)$, obtained as the associated bundle
 \begin{equation*}
 C^{n}(G,K)\times^{Z^{n}(G,K)}{}_{\sigma}Z^n(G,U)
 \end{equation*}
 in the sense of Definition {\rm \ref{def:assocbdle}}, where
 \begin{itemize}\itemsep=0pt
 \item $Z^{n}(G,K)$ acts on ${}_{\sigma}Z^n(G,U)$ by multiplication;
 \item and $C^n(G,K)$ is regarded as a right principal $Z^n(G,K)$-bundle over
 \begin{align*}
 C^n(G,K)/Z^n(G,K) \xrightarrow[\cong]{\quad\partial\quad}
 B^{n+1}(G,K) &\cong B^{n+1}(G,K)\cdot \sigma\\
 &= \pi^{-1}c\subset Z^{n+1}(G,K)
 \end{align*}
 for some fixed $\sigma\in \pi^{-1}c$.
 \end{itemize}

 \item\label{item:relcocfiber3} For abelian $U$
 \begin{equation*}
 Z^n(G,U)
 \le
 {}_{\cE}Z^n(G,U)
 \end{equation*}
 is a Lie subgroup in the sense of {\rm\cite[\emph{Section}~III.1.3, \emph{Definition} 3]{bourb-lie-13}}, fitting into an exact sequence
 \begin{equation}\label{eq:zezim}
 \{1\}
 \to
 Z^n(G,U)
 \xrightarrow{\quad\phantom{\partial}\quad}
 {}_{\cE}Z^n(G,U)
 \xrightarrow{\quad\partial\quad}
 \im(\partial)
 \to
 \{1\}
 \end{equation}
 whose third term $\im(\partial)$ is an open subgroup of $Z^{n+1}(G,K)$.

 \end{enumerate}
\end{Theorem}
\begin{proof}
(1) We have continuous maps
 \begin{equation*}
 {}_{\cE}Z^n(G,U)
 \xrightarrow{\quad\partial\quad}
 Z^{n+1}(G,K)
 \xrightarrow{\quad}
 H^{n+1}(G,K)
 \end{equation*}
 (for non-abelian $U$ too: Remark \ref{re:nccobound}), and the conclusion follows from the discreteness of this last space (see Theorem \ref{th:cocyle-man-n}\,(\ref{item:5})). The same argument (i.e., the discreteness of $H^{n+1}(G,K)$) also proves the openness of
 \begin{equation*}
 \im(\partial) \le Z^{n+1}(G,K)
 \end{equation*}
 in part (\ref{item:relcocfiber3}).

(2) Note first that we do have a map from one of the spaces to the other: the horizontal map in the commutative diagram
 \begin{equation}\label{eq:mapoverpi1c}
 \begin{tikzpicture}[auto,baseline=(current bounding box.center)]
 \path[anchor=base]
 (0,0) node (l) {$C^{n}(G,K)\times^{Z^{n}(G,K)}{}_{\sigma}Z^n(G,U)$}
 +(4,-1) node (d) {$\pi^{-1}c$}
 +(8,0) node (r) {${}_{\pi^{-1}c}Z^n(G,U)$,}
 ;
 \draw[->] (l) to[bend left=6] node[pos=.5,auto] {$\scriptstyle $} (r);
 \draw[->] (l) to[bend right=6] node[pos=.5,auto] {$\scriptstyle $} (d);
 \draw[->] (r) to[bend left=6] node[pos=.5,auto] {$\scriptstyle \partial$} (d);
 \end{tikzpicture}
 \end{equation}
 defined simply as multiplication of a $K$-valued 1-cochain and a relative $U$-valued 1-cocycle. The map
 \begin{align*}
 C^n(G,K)\ni x\xmapsto{\quad}\partial x\cdot c\in B^{n+1}(G,K) c = \pi^{-1}c
 \end{align*}
 is a submersion in the sense of \cite[Section~5.9.1]{bourb-vars-17} (e.g., by Theorem \ref{th:cocyle-man-n}\,(\ref{item:3})), so it splits locally on~$\pi^{-1}(c)$. This then implies that the rightward map in \eqref{eq:mapoverpi1c} is invertible locally on $\pi^{-1}c$.

 To conclude, the proof of Theorem \ref{th:cocyle-man-1}\,(\ref{item:cocyle-man-1}) (and Theorem \ref{th:cocyle-man-n}\,(\ref{item:3})) can be rerun virtually verbatim to show that ${}_{\sigma}Z^n(G,U)$ is a closed analytic submanifold of $C^n(G,U)$.

(3) When $U$ is abelian the three terms of \eqref{eq:zezim} are all groups under multiplication, and the kernel of $\partial$ is precisely $Z^n(G,U)$: this is just about tautological, given the definition of $\partial$. It follows that we indeed have a sequence \eqref{eq:zezim} of (plain, set-theoretic) groups. On the other hand, the fact that $\partial$ is a submersion follows from item (\ref{item:relcocfiber2}), proved previously.
 This completes the proof.
\end{proof}

Theorem \ref{th:relcocfiber}\,(\ref{item:relcocfiber2}) implies in particular that ${}_{\cE}Z^n(G,U)$ are analytic manifolds; their tangent spaces at individual relative cocycles are easy to describe along the lines of Theorem \ref{th:cocyle-man-1}\,(\ref{item:cocyle-man-1-tgz1}) (and Theorem \ref{th:cocyle-man-n}\,(\ref{item:4})), but we need a simple remark in the shape of Lemma \ref{le:equivact}.

The context for the lemma is flexible: it goes through for analytic, or smooth, or $C^r$ manifolds, topological spaces, plain sets, etc. Since it is a slight extension of the discussion in \cite[Section~5.3]{ser-gal} and a simple matter of computational verification, we omit the proof.

\begin{Lemma}\label{le:equivact}
 Let $U$ be a group fitting into a central exact sequence \eqref{eq:gpext} and acting on the left on a space $F$ so that $K\le U$ acts trivially. Let furthermore $G$ be a group operating on the entire structure:
 \begin{itemize}\itemsep=0pt
 \item on all of \eqref{eq:gpext} by automorphisms, via $ G\times U\ni (g,u)\xmapsto{\quad}{}^g u\in U$;
 \item and also on $F$ via $ G\times F\ni (g,f)\xmapsto{\quad}{}^g f\in F$;
 \item so that the action $U\times F\ni (u,f)\xmapsto{\quad} u\triangleright f \in F$
 is $G$-equivariant.
 \end{itemize}
 Given an $\cE$-relative $1$-cocycle $u=(u_g)_g\in {}_{\cE}Z^1(G,U)$,
 \begin{equation*}
 G\times F\ni (g,f)\xmapsto{\quad}
 {}^{g,u} f := u_g\triangleright {}^g f
 \in F
 \end{equation*}
 is again a $G$-action on $F$. \qedhere
\end{Lemma}

\begin{Remarks}\label{res:howserre}\quad
 \begin{enumerate}[(1)]\itemsep=0pt

 \item As in Theorem \ref{th:cocyle-man-1}\,(\ref{item:cocyle-man-1-tgz1}), we will be particularly interested in the case of a~Lie group $U$, acting adjointly on its Lie algebra $F:=\fu$. Clearly, in that case any central subgroup $K\le U$ will operate trivially on $\fu$ and the hypotheses of Lemma \ref{le:equivact} will be met.

 \item The aforementioned \cite[Section~5.3]{ser-gal} is slightly more restrictive than Lemma \ref{le:equivact} only in handling actual, plain cocycles (as opposed to relative ones). The difference is not essential to the verification, given the assumption that $K\le U$ operates trivially on $F$.
 \end{enumerate}
\end{Remarks}

To return to relative cocycles:

\begin{Theorem}\label{th:relcocfiber-tg}
 Assume the hypotheses of Theorem {\rm \ref{th:relcocfiber}}, and let $\fu$ and $\fk$ the Lie algebras of $U$ and $K$ respectively.
 \begin{enumerate}[$(1)$]\itemsep=0pt

 \item\label{item:relcocfiber-tg1} Fix a relative cocycle $u\in {}_{\cE}Z^1(G,U)$, with corresponding $2$-cocycle $\sigma:=\partial u\in Z^2(G,K)$. We have a decomposition
 \begin{equation}\label{eq:z1z2}
 T_{u}\big[{}_{\cE}Z^1(G,U)\big]
 \cong
 \rho_{u} Z^1(G,\fu)
 \oplus
 Z^2(G,\fk),
 \end{equation}
 where
 \begin{itemize}\itemsep=0pt
 \item $\rho_{u}$ indicates right translation of tangent vectors by ${u}\in C(G,U)$;
 \item in the first summand the action of $G$ on $\fu$ is the original one, twisted by the relative cocycle $u$, as in Lemma {\rm \ref{le:equivact}} applied to $F=\fu$;
 \item and the first summand is actually {\it equal} to the tangent space to the fiber ${}_{\cE}Z^1(G,U)$ above $\sigma$ in the fibration \eqref{eq:relznfibers}.
 \end{itemize}

 \item\label{item:relcocfiber-tg2} If $U$ is abelian, the Lie algebra of the Banach Lie group ${}_{\cE}Z^n(G,U)$ fits into an exact sequence
 \begin{equation*}
 \{0\}\to
 Z^n(G,\fu)
 \xrightarrow{\quad}
 \bullet
 \xrightarrow{\quad}
 Z^{n+1}(G,\fk)
 \to\{0\},
 \end{equation*}
 split as a sequence of Banach spaces.
 \end{enumerate}
\end{Theorem}
\begin{proof}
 Part (\ref{item:relcocfiber-tg2}) follows immediately from Theorem \ref{th:relcocfiber}\,(\ref{item:relcocfiber3}).

 The analytic-manifold structure of ${}_{\sigma}Z^1(G,U)$ was already highlighted in the course of the proof of Theorem \ref{th:relcocfiber}, as a variant of Theorem \ref{th:cocyle-man-1}\,(\ref{item:cocyle-man-1}). That same argument, initially proving Theorem \ref{th:cocyle-man-1}\,(\ref{item:cocyle-man-1-tgz1}), also gives the analogous description of the relevant tangent space:
 \begin{equation}\label{eq:turelz1}
 T_u\big[{}_{\sigma}Z^1(G,U)\big] = \rho_{u} Z^1(G,\fu).
 \end{equation}
 To conclude, note that the splitting \eqref{eq:z1z2} follows from the fibration
 \begin{equation*}
 {}_{\cE}Z^n(G,U)\xrightarrow{\quad}\big(\text{open subspace of }Z^2(G,K)\big)
 \end{equation*}
 with fiber ${}_{\sigma}Z^n(G,U)$ of Theorem \ref{th:relcocfiber}\,(\ref{item:relcocfiber2}).
\end{proof}

\begin{Remark}\label{re:martinsame}
 After some processing to account for differences in language and presentation, \cite[Proposition 2.2]{mart-proj} amounts precisely to \eqref{eq:turelz1} for
 \begin{itemize}\itemsep=0pt
 \item the unitary group $U=U(A)$ of a $C^*$-algebra $A$;
 \item with the scalar subgroup $\bS^1\le U$ for our central $K$;
 \item and $G$ acting trivially on $U$.
 \end{itemize}
 The spaces $R(G,\sigma,A)$ of \cite[Section~1.4]{mart-proj} are the individual fibers ${}_{\sigma}Z^1(G,U)$ of the fibration(s)~\eqref{eq:relznfibers}. The remark made there, that two such spaces are homeomorphic for cohomologous cocycles, is also compatible with (and follows from) the description of that fibration as one associated to a principal bundle (see Theorem \ref{th:relcocfiber}\,(\ref{item:relcocfiber2})).
\end{Remark}

The ``analytically-coherent'' mutual conjugation of nearby cocycles of Theorem \ref{th:cocyle-man-1}\,(\ref{item:cocyle-man-1-coboundsplit}) and Theorem \ref{th:cocyle-man-n}\,(\ref{item:6}) goes through in the relative setting as well, with the requisite modifications. The following piece of notation will afford us a uniform statement, valid for either arbitrary $U$ and $n=1$ or abelian $U$ with arbitrary $n\in \bZ_{>0}$.

\begin{Notation}\label{not:actpart}
 Let $G$ be a group acting by automorphisms on another, $U$, and
 \begin{equation*}
 u\in C^n(G,U),\qquad v\in C^{n-1}(G,U)
 \end{equation*}
 two cochains with $n=1$ if $U$ is not abelian \big(in which case $C^{n-1}(G,U)=U$\big). We write
 \begin{equation*}
 v\triangleright_{\partial} u:=
 \begin{cases}
 u+\partial v, &\text{if $U$ is abelian},\\
 G\ni g\mapsto v\cdot u_g \cdot {}^g v^{-1}, &\text{otherwise}.
 \end{cases}
 \end{equation*}
\end{Notation}

\begin{Theorem}\label{th:canconjrel}
 Under the hypotheses of Theorem {\rm \ref{th:relcocfiber}}, every relative cocycle $u$ has a neighborhood $\cU$ in ${}_{\cE}Z^n(G,U)$ admitting an analytic map
 \begin{equation*}
 \cU\ni {u'}
 \xmapsto{\quad}
 (v_{u'},\ w_{u'})\in C^{n-1}(G,U)\times C^n(G,K)
 \end{equation*}
 such that
 \begin{equation*}
 (v_{u},w_{u})=(1,1)
 \qquad\text{and}\qquad
 v_{u'}\triangleright_{\partial} (u'\cdot w_{u'}) = u,\qquad \forall u'\in \cU.
 \end{equation*}
\end{Theorem}
\begin{proof}
 This essentially stitches together two previous results. Set $\sigma:=\partial u\in Z^{n+1}(G,K)$.
 Relative cocycles $u'$ sufficiently close to $u$ will have differentials close to $\sigma$, whence, by Theorem~\ref{th:cocyle-man-n}\,(\ref{item:6}), an analytic map
 \begin{equation*}
 \cU\ni u'\mapsto w_{u'}\in C^n(G,K)
 \end{equation*}
 with
 \begin{equation}\label{eq:wfirst}
 w_{u}=1
 \qquad\text{and}\qquad
 u'\cdot w_{u'} \in {}_{\sigma}Z^n(G,U),\qquad \forall u'\in \cU
 \end{equation}
 for some small neighborhood $\cU\ni u$ in ${}_{\cE}Z^n(G,U)$.

 We can now focus exclusively on the fiber ${}_{\sigma}Z^n(G,U) = \partial^{-1}\sigma$ of \eqref{eq:relznfibers}. Essentially the same proof as that of Theorem \ref{th:cocyle-man-1}\,(\ref{item:cocyle-man-1-coboundsplit}) (and its abelian counterpart Theorem \ref{th:cocyle-man-n}\,(\ref{item:6})) will then provide an analytic map
 \begin{equation*}
 \cU_0\ni u'\mapsto v_{u'}\in C^{n-1}(G,U)
 \end{equation*}
 with
 \begin{equation*}
 v_{u}=1
 \qquad\text{and}\qquad
 v_{u'}\triangleright_{\partial} u' = u,\qquad \forall u'\in \cU_0.
 \end{equation*}
 for a neighborhood $\cU_0\ni u$ in ${}_{\sigma}Z^n(G,U)$.

 Shrinking the initial $\cU$ so that all $u'\cdot w_{u'}$ in \eqref{eq:wfirst} belong to $\cU_0$, we have our conclusion.
\end{proof}

It is worth recording a consequence (of the proof) pertaining to a single fiber ${}_{\sigma}Z^n(G,U)$ of the fibration \eqref{eq:relznfibers}.

\begin{Corollary}\label{cor:orbisfibrel}
 Under the hypotheses of Theorem {\rm \ref{th:relcocfiber}}, fix $\sigma\in Z^{n+1}(G,K)$.
 \begin{enumerate}[$(1)$]\itemsep=0pt
 \item The orbits in ${}_{\sigma}Z^n(G,U)$ under the action $\triangleright_{\partial}$ of $C^{n-1}(G,U)$ are open.

 \item For each element $u\in {}_{\sigma}Z^n(G,U)$ the corresponding orbit map
 \begin{equation*}
 C^{n-1}(G,U)\ni v\xmapsto{\quad\mathrm{orb}_u}v\triangleright_{\partial}u\in {}_{\sigma}Z^n(G,U)
 \end{equation*}
 is a principal fibration over the orbit of $u$.

 \item\label{item:orbisfib} The isotropy group
 \begin{equation*}
 \mathrm{Fix}_u:=\big\{v\in C^{n-1}(G,U)\mid v\triangleright_{\partial}u = u\big\}\le C^{n-1}(G,U)
 \end{equation*}
 is a closed Lie subgroup.
 \end{enumerate}
\end{Corollary}
\begin{proof}

 The openness of the orbits was noted in the course of the proof of Theorem \ref{th:canconjrel}, and the local splitting of the orbit map ensures that the orbit map is a surjective submersion. The conclusion then follows from \cite[Section~III.1.6, Proposition 12]{bourb-lie-13}.
\end{proof}

In particular, specializing Corollary \ref{cor:orbisfibrel} to trivial actions and trivial $K\le U$, we obtain

\begin{Corollary}\label{cor:centcpct}
 The centralizer of a compact subgroup in a Banach Lie group is a closed Lie subgroup. \qedhere
\end{Corollary}

\begin{Remark}
 See \cite[Section~VI.3\,(A)]{om-inflie-bk} for a version of Corollary \ref{cor:centcpct} for groups of diffeomorphisms of smooth, compact manifolds (in place of our $U$); such groups are Lie in a broader sense, being modeled on more general topological (non-Banach) vector spaces.
\end{Remark}

When the ambient Banach Lie group $U$ is in fact finite-dimensional, any centralizer $G$ (of any subset) is a closed subgroup and hence a Lie subgroup \cite[Section~III.8.2, Theorem 2]{bourb-lie-13}. This latter result is no longer valid for infinite-dimensional Lie groups (i.e., closed no longer implies Lie):
\begin{itemize}\itemsep=0pt
\item As noted in Section~\ref{subse:banmf}, \cite[Section~III.1.3, Definition 3]{bourb-lie-13} requires (via \cite[Section~5.8.3]{bourb-vars-17}) that the embedding $\fg\le \fu$ of Lie algebras attached to a Lie-subgroup embedding $G\le U$ split. Consider a non-split closed linear embedding $F\le E$ of Banach spaces (e.g., the space $c_0$ of real-valued sequences converging to 0, not complemented in the space $\ell^{\infty}$ of bounded real-valued sequences \cite[Example 14-4-9]{wil-tvs}). The subgroup
 \begin{gather*}
 (F,+)\subset (E,+)
 \end{gather*}
 is then only quasi-Lie, but not Lie (see Convention \ref{cv:qlie}).
\item Worse still, \cite[Remark IV.4 (d)]{neeb-inf} gives an example of an additive closed subgroup $G\le U$ of a Hilbert space whose underlying ``Lie algebra candidate''~\cite[Corollary IV.3]{neeb-inf}
 \begin{equation*}
 \fg:=\{x\in \fu\mid \exp(\bR x)\subseteq \fu\}
 \end{equation*}
is trivial, but which is nevertheless not discrete.

 In this case, the closed subgroup in question is not even quasi-Lie.
\end{itemize}

The first item above can be adapted to give a counterexample to Corollary~\ref{cor:centcpct} in the absence of both compactness and finite-dimensionality.

\begin{Example}\label{ex:notlie}
 Consider a non-split closed embedding $F\le E$ of Banach spaces, regarded as additive Lie groups. This then gives a Lie-group morphism $E\xrightarrow{\quad\pi\quad} E/F$ whose kernel is not a Lie subgroup.

 Next, note that
 \begin{equation*}
 E\oplus E/F\ni (x,y)\xmapsto{\quad\alpha\quad}(x,\pi(x)+y)\in E\oplus E/F
 \end{equation*}
 is an {\it auto}morphism whose fixed-point set $F\oplus E/F$ is not Lie. Finally, for the semidirect product
 \begin{equation*}
 G:=\left(E\oplus E/F\right)\rtimes \bZ
 \end{equation*}
 with a generator $\sigma\in \bZ$ acting via $\alpha$. The centralizer (and normalizer) $(F\oplus E/F)\times \bZ\subset G$ of~$\bZ=\langle\sigma\rangle$ is then not a Lie subgroup in the sense of \cite[Section~III.1.3, Definition 3]{bourb-lie-13}.
\end{Example}

\begin{Remark}\label{re:centnotbad}
 Counterexamples to Corollary \ref{cor:centcpct} cannot be much more pathological than Example \ref{ex:notlie}: for any Banach Lie group $U$, the fixed-point set $U^{\alpha}$ of an automorphism ${\alpha\in \mathrm{Aut}(U)}$ is always at least quasi-Lie.

 This is easily seen from the fact that $\alpha$ induces an automorphism of the Lie algebra ${\fu\!:=\!\operatorname{Lie}(U)}$ so that the exponential map $\exp\colon\fu\to U$ is equivariant, so a small identity neighborhood of~$1\in U^{\alpha}$ can be identified via the exponential map with a small neighborhood of $0$ in the Lie algebra $\fu^{\alpha}$ of $\alpha$-invariant vectors.

 All of this goes through for arbitrary automorphism groups of $U$ in place of the single $\alpha$: the fixed-point set $U^H\le U$ is a closed quasi-Lie subgroup with Lie algebra $\fu^H$. See also \cite[Proposition 5.54\,(ii)]{hm4}, which states this for {\it linear} Lie groups.
\end{Remark}

A supplement to Corollary \ref{cor:centcpct}:

\begin{Proposition}\label{pr:normcpct}
 The normalizer $N_U(G)\le U$ of a compact subgroup $G\le U$ in a Banach Lie group is a closed Lie subgroup.
\end{Proposition}
\begin{proof}
 It will be enough to prove this of the identity component $N_0$ of $N:=N_U(G)$. Set $C:=C_U(G)\cap N_0$,
 i.e., the centralizer of $G$ in $N_0$. Since the entire centralizer $C_U(G)$ is contained in $N$ to begin with, and is a Lie subgroup of $G$ by Corollary \ref{cor:centcpct}, $C\le G$ too is Lie. Furthermore, we have~$N_0=CG$ by \cite[Theorem 2]{iw}, so that $N_0$ is the image of
 \begin{equation}\label{eq:multmap}
 C\times G\ni (c,g)\xmapsto{\quad} cg\in U.
 \end{equation}
 $G \le U$ is a finite-dimensional Lie subgroup \cite[Theorem IV.3.16]{neeb-lc}, so the domain of \eqref{eq:multmap} is a Lie group and that morphism is plainly analytic. Furthermore, its differential is simply the addition map
 \begin{equation*}
 T_1C\times T_1G\xrightarrow{\quad + \quad}T_1U,
 \end{equation*}
 and hence its image is the space
 \begin{equation}\label{eq:plusmap}
 T_1C + T_1G\le T_1U\quad\text{(the sum need not be direct)}.
 \end{equation}
 That subspace is
 \begin{itemize}\itemsep=0pt
 \item closed, being the sum of the closed subspace $T_1C\le T_1U$ and the finite-dimensional $T_1G$;
 \item and complemented, as follows from the fact that $T_1C$ is complemented (see Corollary~\ref{cor:centcpct} again) and that finite-dimensional subspaces are always complemented \cite[Theorem \mbox{7-3-6}]{wil-tvs}.
 \end{itemize}
 The kernel of the addition map \eqref{eq:plusmap} can be identified with the intersection $T_1C\cap T_1G$ via
 \begin{equation*}
 T_1C\cap T_1G\ni v\xmapsto{\quad} (v,-v)\in T_1C\oplus T_1G.
 \end{equation*}
 That intersection is also the Lie algebra of the Lie group $C\cap G = \text{kernel of \eqref{eq:multmap}}$,
 also regarded as a subgroup of $C\times G$ by the corresponding ``anti-diagonal'' morphism
 \begin{equation*}
 C\cap G\ni g\xmapsto{\quad} \big(g,g^{-1}\big)\in C\times G
 \end{equation*}
 (a morphism, since $C\cap G$ is abelian). Indeed, the argument of Remark \ref{re:centnotbad} shows that
 \begin{equation*}
 T_1C = (T_1U)^{\operatorname{Ad}(G)}
 \end{equation*}
 (fixed points under the adjoint action), while by the same token
 \begin{equation*}
 T_1(C\cap G) = (T_1G)^{\operatorname{Ad}(G)} = (T_1U)^{\operatorname{Ad}(G)}\cap T_1G = T_1C\cap T_1G.
 \end{equation*}
 In conclusion:
 \begin{itemize}\itemsep=0pt
 \item the differential of the morphism \eqref{eq:multmap} vanishes precisely along the tangent space of its kernel,
 \item and hence that map induces an immersion $(C\times G)/C\cap G\to U$ (the quotient of \cite[Section~III.1.6, Proposition 11\,(ii)]{bourb-lie-13}, say) onto a closed subgroup ${CG\le U}$,
 \item so that the image of the differential is a complemented subspace of $T_1U$.
 \end{itemize}
 This is sufficient to give $CG=N_0\le U$ its Lie-subgroup structure, finishing the proof.
\end{proof}

\begin{Remark}\label{re:isothmbourb}
 The argument employed in the proof of Proposition \ref{pr:normcpct} is by no means unfamiliar: \cite[Section~III.3.8, Proposition 30]{bourb-lie-13}, for instance, similarly considers the multiplication map $G_0\times G_1\to U$ for Lie subgroups $G_i\le U$ that mutually centralize each other, its differential, etc. The difference there is that the ambient group $U$ is assumed finite-dimensional, whereas in Proposition \ref{pr:normcpct} the compactness of one of the groups $G_i$ makes up for the infinite dimensionality.
\end{Remark}

\section{Applications to operator-algebra morphisms}\label{se:oaapp}

\subsection{Manifolds of morphisms}\label{subse:manmor}

The material above affords a quick proof for the following result, concerning spaces of Banach- or $C^*$-algebra morphisms.

\begin{Theorem}\label{th:castmorcross}
 Let $A$ and $B$ be two unital $C^*$-algebras, with $A$ finite-dimensional.
 \begin{enumerate}[$(1)$]\itemsep=0pt
 \item The space $C^*(A,B)\subset L(A,B)$ is a closed submanifold.
 \item\label{item:a2brefreport} The orbits therein under conjugation by the unitary group $U(B)$ are all open in $C^*(A,B)$.
 \item\label{item:castcross} For every $\varphi\in C^*(A,B)$ there is a neighborhood $\cU\subset C^*(A,B)$ of $\varphi$ and an analytic map
 \begin{equation*}
 \cU\ni \theta\xmapsto{\quad}u_{\theta}\in U(B)
 \end{equation*}
 such that $u_{\varphi}=1$ and $\theta = u_{\theta}\cdot \varphi\cdot u_{\theta}^*$, $\forall \theta\in \cU$.
 \item For any morphism $\varphi\in C^*(A,B)$ the orbit map
 \[U(B)\ni u\mapsto u\varphi u^*\] is a fibration over the orbit.

 \end{enumerate}
\end{Theorem}
\begin{proof}
 All of this follows immediately from Corollary \ref{cor:morsplit}, applied to $G:=U(A)$ and $U:=U(B)$: note that a finite-dimensional $C^*$-algebra is the span of its (compact!) unitary group, and hence that the $U(B)$-orbits in $C^*(A,B)$ under conjugation can be identified with some of the $U(B)$-orbits in $\Hom(U(A), U(B))$.
\end{proof}

\begin{Remark}\label{re:wo}
 \cite[Proposition 5.2.6]{wo} addresses the analogue for spaces of {\it projections} in $B$ rather than morphisms $A\to B$; given the ($B$-natural) identification
 \begin{equation*}
 \begin{tikzpicture}[auto,baseline=(current bounding box.center)]
 \path[anchor=base]
 (0,0) node (l) {$\left(\text{projections }\in B\right)$}
 +(8,0) node (r) {$C^*\big(\bC^2,B\big)$,}
 +(4,0) node (m) {$\cong$}
 ;
 \draw[->] (l) to[bend left=10] node[pos=.5,auto] {$\scriptstyle p\mapsto \left(\text{morphism sending }(1,0)\mapsto p\right)$} (r);
 \draw[->] (r) to[bend left=10] node[pos=.5,auto] {$\scriptstyle \varphi\mapsto \varphi(1,0)$} (l);
 \end{tikzpicture}
 \end{equation*}
 that result is essentially Theorem \ref{th:castmorcross}\,(\ref{item:castcross}) specialized to $A=\bC^2$. In \cite[Proposition 5.2.6]{wo}, the analogue of the map $\theta\mapsto u_{\theta}$ of Theorem \ref{th:castmorcross} is only required to be continuous, but it is clear by direct examination that the proof in fact provides an {\it analytic} map.
\end{Remark}

There is also a Banach-algebra analogue of Theorem \ref{th:castmorcross}, which we record here for completeness. By way of a preamble, recall that an algebra is {\it semisimple} if it is a finite product of matrix algebras over division rings (e.g., \cite[Theorem/Definition 2.5 and Theorem 3.5]{lam} or \cite[Section~3.1, Definition and Section~3.5, Theorem]{pierce}). Since we will only apply the term to finite-dimensional real or complex algebras, this means matrix algebras over $\bR$, $\bH$ (the quaternions), or $\bC$.

\begin{Remark}\label{re:ssban}
 It is customary, in the Banach-algebra literature, to see the term `semisimple' used somewhat differently: in both \cite[Definition 1.5.1]{dales} and \cite[Section~2.2]{dael}, for instance, the term, applied to a unital Banach algebra $A$, means that the {\it Jacobson radical} $\operatorname{rad}A$ vanishes.

 Recall (for instance from \cite[Section~4]{lam} or \cite[Section~4.3]{pierce}) that $\operatorname{rad}A$ is the intersection of all maximal left ideals or, equivalently, all maximal right ideals; those sources contain many other characterizations of the Jacobson radical.

 In short, the semisimplicity of \cite[Definition 1.5.1]{dales} is the {\it Jacobson- $($or J-$)$semisimplicity} of~\cite[Definition 4.7]{lam}. Our language here follows the latter source rather than the former.
\end{Remark}

\begin{Theorem}\label{th:banmorcross}
 Let $A$ and $B$ be two unital Banach algebras, with $A$ finite-dimensional and semisimple.
 \begin{enumerate}[$(1)$]\itemsep=0pt
 \item\label{item:banisman} The space $\cat{BAlg}(A,B)\subset L(A,B)$ is a closed submanifold.

 \item\label{item:banopenorb} The orbits therein under conjugation by the group $B^{\times}$ of invertible elements are all open in~$\cat{BAlg}(A,B)$.

 \item\label{item:bansect} Every morphism $\varphi\colon A\to B$ has a neighborhood $\cU\subset \cat{BAlg}(A,B)$ admitting an analytic map
 \begin{equation*}
 \cU\ni \theta\xmapsto{\quad}u_{\theta}\in B^{\times}
 \end{equation*}
 such that $u_{\varphi}=1$ and $\theta = u_{\theta}\cdot \varphi\cdot u_{\theta}^{-1}$, $\forall \theta\in \cU$.

 \item\label{item:banfib} For any morphism $\varphi\in \cat{BAlg}(A,B)$ the orbit map
 \begin{equation*}
 B^{\times}\ni u\mapsto u\cdot \varphi\cdot u^{-1}
 \end{equation*}
 is a fibration over the orbit.

 \end{enumerate}
\end{Theorem}
\begin{proof}
 Parallel to that of Theorem \ref{th:castmorcross}, except that now Corollary \ref{cor:morsplit} is to be applied to~${U\!:=\!B^{\times}}$ and $G$ a maximal compact subgroup of the (finite-dimensional) Lie group $A^{\times}$: since $A$ is a~product of matrix algebras over $\bR$, $\bC$ or $\bH$, $G$ is a product of orthogonal, unitary and (compact) symplectic \cite[Exercise 7.4 and discussion preceding it]{fh} groups.

 Because the (real, complex, or quaternionic) unitary group of a matrix algebra spans the latter over $\bR$, the semisimplicity of $A$ ensures that $\mathrm{span}_{\bR}(G)=A$;
 we can then conclude as before.
\end{proof}

\begin{Remarks}\label{res:wogen}\quad
 \begin{enumerate}[(1)]\itemsep=0pt
 \item Spaces of morphisms have been studied extensively in the operator-algebra literature from this same differential-geometric perspective. Papers we refer to below include~\cite{acs,cg-average,cg-amen}, and {\it their} references will provide the reader with ample material.
 \item\label{item:acs} \cite[Proposition 2.2]{acs} is a version of Theorem \ref{th:castmorcross} which
 \begin{itemize}\itemsep=0pt
 \item on the one hand gives more precise numerical estimates for how close to 1 the element $u_{\theta}$ is depending on $\|\varphi-\theta\|$;
 \item while on the other specializes to $B=B(H)$ (operators son some Hilbert space $H$) and conjugates by the general linear group $B^{\times}$ in place of $U(B)$;
 \item and is not concerned with the analytic structure of the orbit.
 \end{itemize}
 That proof adapts easily enough to recover the various missing pieces: the extension to arbitrary $B$ follows, for instance, from \cite[Remark~2.4]{acs}, while the switch from $B^{\times}$ to $U(B)$ can be effected by a polar-decomposition argument. In short, one can essentially read Theorem \ref{th:castmorcross} into \cite[Proposition 2.2]{acs}. The point here, rather, is to obtain results such as Theorem~\ref{th:castmorcross} and Theorem \ref{th:banmorcross} uniformly, as consequences of the preceding discussion on the Banach--Lie-group-valued cohomology of compact groups.

 \item There is an alternative approach to Theorem \ref{th:banmorcross}, along the lines of \cite[Proposition 8.2.2 and Theorem 8.2.3]{run-lec} (in turn based on \cite[Proposition 4.1 and Theorem 4.5]{cg-amen}).

 In those references the domain algebra $A$ is assumed {\it amenable} in the sense of \cite[Definition~2.1.9]{run-lec}, and the codomain $B$ is assumed to be dual (i.e., the dual of a Banach space). Because for us $A$ is in fact finite-dimensional, the proof of \cite[Theorem 8.2.1]{run-lec} goes through for arbitrary $B$ upon dropping the `$**$' superscripts (indicating double duals) throughout.

 With this in place, \cite[Proposition 8.2.2 and part of Proposition 8.2.3]{run-lec} would be analogous to Theorem \ref{th:banmorcross}.
 \end{enumerate}
\end{Remarks}

Even when everything in sight is finite-dimensional, semisimplicity is essential in Theorem~\ref{th:banmorcross}: not only will $\cat{BAlg}(A,B)$ not, in general, be a submanifold of $L(A,B)$, but it will in fact not even be a (topological) manifold. Example \ref{ex:notmfld-notss} illustrates this.

\begin{Example}\label{ex:notmfld-notss}
 We work over the real numbers. Let $A:=\bR[\varepsilon]/\big(\varepsilon^2\big)$ and $B:=M_2(\bR)$. Since everything is finite-dimensional the Banach-algebra norms in question are all equivalent, so we need not bother fixing them.

 The morphisms $A\to B$ fall into two classes:
 \begin{itemize}\itemsep=0pt
 \item those that fail to annihilate $\varepsilon$, and hence identify it with a rank-1 nilpotent matrix in~$M_2(\bR)$;
 \item and the only one that does (annihilate $\varepsilon$):
 \begin{equation}\label{eq:degmor}
 1\mapsto 1,\qquad \varepsilon\mapsto 0.
 \end{equation}
 \end{itemize}
 The set of rank-1 nilpotent $2\times 2$ matrices can be identified with pairs consisting of
 \begin{itemize}\itemsep=0pt
 \item a line $\ell\subset \bR^2$ (the kernel of the matrix);
 \item and a non-zero morphism $\bR^2/\ell\to \ell$.
 \end{itemize}
 This is nothing but the tangent bundle to the circle $\bS^1\!\cong\! \bR\bP^1$ (\cite[Theorem 3.5]{3264}, \cite[Lemma~4.4]{ms-charclass}, etc.) with the zero section removed, i.e., $\bS^1\!\times \bR^{\times}$. The addition of the degenerate morphism~\eqref{eq:degmor} glues the two copies of $\bS^1\times \bR_{>0}$ together, producing two open disks attached along a common origin. Plainly, this is not a manifold.
\end{Example}

That nearby morphisms need not be conjugate absent semisimplicity is already in evidence in Example \ref{ex:notmfld-notss}: the unique morphism annihilating the radical, in that example, is in the closure of the set of morphisms which do not. It turns out this is a rather general phenomenon.

\begin{Proposition}\label{pr:clorb}
 Let $\varphi\colon A\to B$ be a morphism in $\cat{BAlg}$, with $A$ finite-dimensional.
 \begin{enumerate}[$(1)$]\itemsep=0pt
 \item If $\varphi$ annihilates $\operatorname{rad}A$, the orbit of $\varphi$ under conjugation by $B^{\times}$ is closed.
 \item The converse holds if $B$ is finite-dimensional and semisimple.
 \end{enumerate}
\end{Proposition}
\begin{proof}
 We prove the two claims in the order in which they were stated.

(1) Every element in the orbit of $\varphi$ will still annihilate the radical, hence so do the morphisms in the closure of that orbit. This reduces the problem to morphisms defined on the {\it semisimple} Banach algebra $\overline{A}:=A/\operatorname{rad}A$ and the claim follows from the openness of the $B^{\times}$-orbits in~$\cat{BAlg}\big(\overline{A},B\big)$ (Theorem \ref{th:banmorcross}\,(\ref{item:banopenorb})), because the orbits partition the space of morphisms.

(2) As for the converse, note first that in finite-dimensional, semisimple algebras nilpotent elements are conjugate to their (real, say) non-zero scalings. It follows that if $J\subset A$ is the top non-zero power of the radical, $\varphi$ will be conjugate to a morphism that scales every element of $J$ by an arbitrary $r\in \bR^{\times}$. The closure of the orbit of $\varphi$ thus contains morphisms that annihilate~$J$, and we can proceed by induction on
 \begin{equation*}
 \max\{n\mid(\operatorname{rad}A)^n\ne 0\}.
 \end{equation*}
 The proof is complete.
\end{proof}

What delivers Theorems \ref{th:castmorcross} and~\ref{th:banmorcross}, ultimately, is the rigidity inherent to compact groups (which in turn ensures that their vector-space-valued cohomology vanishes, provides the requisite Banach-space splitting through Lemma~\ref{le:cansplit}, etc.). That rigidity makes an appearance in the proofs only obliquely, via the fact that semisimple algebras are spanned by compact subgroups. It is perhaps pertinent at this point to spell out the converse:

\begin{Lemma}\label{le:mustbess}
 A finite-dimensional real algebra $A$ is spanned by a compact subgroup of $A^{\times}$ if and only if it is semisimple.
\end{Lemma}
\begin{proof}
 We have already noted the backward implication in the course of the proof of Theorem~\ref{th:banmorcross}. Conversely, suppose $A$ is spanned by some compact subgroup $K\le A^{\times}$. Setting $J:=\operatorname{rad}A$, the short exact sequence
 \begin{equation}\label{eq:algseq}
 \{0\}\to J\xrightarrow{\quad}A\xrightarrow{\quad}\overline{A}\to \{0\}
 \end{equation}
 induces an extension
 \begin{equation*}
 \{1\}\to J+1\xrightarrow{\quad\phantom{\pi}\quad} A^{\times}\xrightarrow{\quad\pi\quad} \overline{A}^{\times}\to \{1\}
 \end{equation*}
 (an element of $A$ is invertible if and only if it is invertible modulo $J$ \cite[Proposition 4.8]{lam}, so the rightmost map is indeed onto). Because $J+1$ is a successive extension of copies of $(\bR,+)$, $K$ cannot intersect it and is thus mapped by $\pi$ isomorphically onto a compact subgroup of $\overline{A}^{\times}$.

 By the {\it Malcev--Wedderburn theorem} \cite[Section~11.6, Theorem]{pierce}, the sequence \eqref{eq:algseq} splits essentially uniquely, in the sense that
 \begin{itemize}\itemsep=0pt
 \item there are subalgebras $A'\le A$ mapped isomorphically onto $\overline{A}$ by $\pi$;
 \item and such algebras are all mutually $(J+1)$-conjugate.
 \end{itemize}
 There is no harm in assuming that $K\le A^{\times}$ is {\it maximal} compact \cite[Theorem~VII.1.2]{bor-symsp}. But now a subalgebra $A'\le A$ splitting \eqref{eq:algseq} contains an isomorphic copy $K'$ of the maximal compact subgroup $K\le A^{\times}$; because maximal compact subgroups are mutually conjugate (see \cite[Theorem~VII.1.2]{bor-symsp} again), some conjugate
 \begin{equation*}
 A'':=gA'g^{-1},\qquad g\in A^{\times}
 \end{equation*}
 contains $K$. This means in particular that the span of $K$ is $A''$, which was by construction a~{\it semisimple} subalgebra of $A$. This, then, proves the latter's semisimplicity.
\end{proof}

\cite[Theorem 8.2.3]{run-lec} and \cite[Theorem 4.5]{cg-amen} give more information on spaces of Banach-algebra morphisms, identifying their tangent bundles in the spirit of Theorem \ref{th:cocyle-man-1}\,(\ref{item:cocyle-man-1-tgz1}). This too follows from the preceding material, but requires a brief detour.

\subsection{Background on Hochschild cohomology}\label{subse:hoch}

This will be little more than a pointer to the extensive literature on Hochschild cohomology, and an opportunity to fix some notation.

For the purely algebraic theory the reader can consult, for instance, \cite[Section~1.9]{dales}, \cite[Chapter 11; especially Section~11.1]{pierce} or \cite[Sections~4 and 5]{ginz-ncg}. Given an algebra $A$ (unital, for us) and an $A$-bimodule $M$, denote by
\begin{equation}\label{eq:cochains}
 C^n(A,M):=\mathrm{Hom}\big(A^{\otimes n},M\big) \cong \left\{\text{multilinear maps }A^n\to M\right\}
\end{equation}
the set of {\it $n$-cochains} ($M$-valued, attached to the bimodule structure). In the functional-analytic setup (\cite[Section~1]{john-coh}, \cite[Section~2.8]{dales}, \cite[Section~7]{dael}) $A$ is a Banach algebra, $M$ is a {\it Banach bimodule} \cite[Definition 2.6.1]{dales}, and the multilinear maps are required to be bounded. Equivalently~\cite[Theorem 2.9]{ryan}, the middle term of \eqref{eq:cochains} ought to be $\cat{Ban}\big(A^{\otimes n},M\big)$,
where the left-hand term denotes the {\it projective Banach-space tensor product} of \cite[Section~2.1]{ryan}.

We then have a cochain complex
\begin{equation}\label{eq:hochcoc}
 0\xrightarrow{}
 E\cong C^0(A,E)
 \xrightarrow{\delta^0}
 C^1(A,E)
 \xrightarrow{\delta^1}
 \cdots \xrightarrow{\delta^{n-1}}
 C^n(A,E)
 \xrightarrow{\delta^n}
 \cdots
\end{equation}
defined by
\begin{align*}
 \delta^n f (a_1,a_2,\dots,a_{n+1}):={}& a_1 f(a_2,\dots,a_{n+1})\\
 &+\sum_{i=1}^{n} (-1)^{i} f(a_1,\dots,a_{i-1},a_i a_{i+1},a_{i+2},\dots,a_{n+1})\\
 &+(-1)^{n+1} f(a_1,\dots,a_n)a_{n+1}.
\end{align*}

Then, as for groups (cf.\ Section~\ref{subse:coh}), define
\begin{itemize}\itemsep=0pt
\item the (vector or Banach) space
 \begin{equation*}
 Z^n(A,M):=\ker\big(C^n(A,M)\xrightarrow{\delta^n}C^{n+1}(A,M)\big)
 \end{equation*}
 of {\it $n$-cocycles};
\item the (vector or Banach) space
 \begin{equation*}
 B^n(A,M):=\im\big(C^{n-1}(A,M)\xrightarrow{\delta^{n-1}}C^{n}(A,M)\big)
 \end{equation*}
 of {\it $n$-coboundaries};
\item and the {\it $n^{\rm th}$ Hochschild cohomology}
 \begin{equation*}
 H^n(A,M):=Z^n(A,M)/B^n(A,M).
 \end{equation*}
\end{itemize}

\subsection{Group/Hochschild cocycles as tangent spaces}\label{subse:cocytg}

We can now return to the main thread, supplementing Theorems \ref{th:castmorcross} and \ref{th:banmorcross} with the requisite tangent-bundle information. The statement of Theorem \ref{th:tgsp} requires one more prefatory discussion, on how Hochschild cohomology meshes with $*$-structures.

To wit, suppose $E$ is a complex Banach space equipped with
\begin{itemize}\itemsep=0pt
\item a $*$-structure (as usual: an involutive conjugate-linear continuous operation);
\item and an $A$-bimodule structure for a $C^*$-algebra $A$;
\item so that the two are compatible in the guessable sense:
 \begin{equation*}
 (am)^* = m^* a^*
 \qquad\text{and}\qquad
 (ma)^* = a^* m^*,\qquad \forall a\in A,\quad m\in E.
 \end{equation*}
\end{itemize}
There is then a $*$-structure on cochains given by
\begin{equation*}
 f^*(a_1, \dots, a_n):=f(a_n^*, \dots, a_1^*)^*,\qquad \forall f\in C^n(A,E)
\end{equation*}
(note the argument order reversal), which intertwines the differentials $\delta^i$ of \eqref{eq:hochcoc}. In particular, it also induces $*$-structures on cocycle and coboundary spaces, and allows us to speak, for instance, of {\it self-adjoint} elements therein:
\begin{equation*}
 Z_{sa}^n(A,E):=\{f\in Z^n(A,E)\mid f^*=f\},
\end{equation*}
and so on.

\begin{Theorem}\label{th:tgsp}
 Let $A$ and $B$ be unital operator algebras, either Banach or $C^*$, with $A$ finite-dimensional and, in the Banach case, semisimple.
 \begin{enumerate}[$(1)$]\itemsep=0pt
 \item\label{item:bantg} For a unital Banach-algebra morphism $\varphi\colon A\to B$, the tangent space
 \begin{equation}\label{eq:tgban}
 T_{\varphi}\cat{BAlg}(A,B)\subset T_{\varphi}\cat{Ban}(A,B)\cong \cat{Ban}(A,B)
 \end{equation}
 is precisely the space $Z^1(A,B)$ of Hochschild $1$-cocycles for the $A$-bimodule structure on $B$ induced by $\varphi$.

\item\label{item:casttg} For a unital $C^*$ morphism $\varphi\colon A\to B$, the tangent space
 \begin{equation*}
 T_{\varphi}C^*(A,B)\subset T_{\varphi}\cat{Ban}(A,B)\cong \cat{Ban}(A,B)
 \end{equation*}
 is the space $Z_{sa}^1(A,B)$ of self-adjoint Hochschild $1$-cocycles for the $A$-bimodule structure on~$B$ induced by $\varphi$.
\end{enumerate}
\end{Theorem}
\begin{proof}
(1) As in the proof of Theorem \ref{th:banmorcross}, we identify the Banach-algebra morphisms $A\to B$ with {\it some} of the Lie-group morphisms
 \begin{equation*}
 G\to B^{\times},\qquad G\subset A^{\times}\quad\text{ maximal compact}.
 \end{equation*}
 Since the Lie algebra of $B^{\times}$ is $B$, Theorem \ref{th:cocyle-man-1}\,(\ref{item:cocyle-man-1-tgz1}) shows that \eqref{eq:tgban} consists of the maps
 \begin{equation}\label{eq:f'}
 G\ni g\xmapsto{\quad f'\quad} f(g)\varphi(g)\in B,\qquad f\in Z^1(G,B),
 \end{equation}
 extended by linearity to all of $A$, where the $G$-action on $B$ is given by conjugation via $\varphi$:
 \begin{equation*}
 g\triangleright b = \varphi(g)b\varphi(g)^{-1},\qquad \forall g\in G,\quad b\in B.
 \end{equation*}
 That one can indeed extend $f'$ linearly to $A$ follows from the fact that it is a tangent vector to $\cat{BAlg}(A,B)$ (a manifold by Theorem \ref{th:banmorcross}\,(\ref{item:banisman})), so can be expressed as a limit of differences of Banach-algebra morphisms and hence linear maps.

 The group-cocycle condition
 \begin{equation*}
 f(gh) = f(g) + g\triangleright f(h),\qquad \forall g,h\in G
 \end{equation*}
 translates to the Hochschild-cocycle condition
 \begin{equation*}
 f'(gh) = gf'(h) + f'(g)h,\qquad \forall g,h\in G
 \end{equation*}
 for $f'$ as in \eqref{eq:f'}, so that by linearity $f'$ is indeed a Hochschild cocycle.

 Conversely, any Hochschild cocycle $f'\in Z^1(A,B)$ gives rise to a group cocycle
 \begin{equation*}
 G\ni g\xmapsto{\quad f\quad} f'(g)\varphi(g)^{-1}\in B.
 \end{equation*}
 In short, the present statement is a reformulation of Theorem \ref{th:cocyle-man-1}\,(\ref{item:cocyle-man-1-tgz1}) in the very specific case of Lie-group morphisms $G\to B^{\times}$ extensible to Banach-algebra morphisms.

(2) The one missing ingredient, as compared to part (\ref{item:casttg}), is the self-adjointness. For this, recall first that the Lie algebra of the unitary group $U(B)$ is the space of skew-hermitian elements~\cite[Corollary 15.22]{upm-ban}, i.e., those satisfying $b^*=-b$. Then, we can apply Theorem~\ref{th:cocyle-man-1}\,(\ref{item:cocyle-man-1-tgz1}) as before, noting that in the presence of
 \begin{equation*}
 0 = f(1) = f(g\cdot g{-1}) = f(g) + \varphi(g)f\big(g^{-1}\big)\varphi(g)^*
 \end{equation*}
 (a consequence of the cocycle condition for $f\in Z^1(U(A),B)$), the conditions
 \begin{itemize}\itemsep=0pt
 \item $f(g)^* = f(g)$, expressing the fact that $f$ takes values in the Lie algebra $\operatorname{Lie}(U(B))$;
 \item and
 \begin{equation*}
 f(g^*)\varphi(g^*) = \varphi(g)^* f(g)^* = (f(g)\varphi(g))^*,
 \end{equation*}
 expressing the fact that the Hochschild cocycle $f':=f\cdot \varphi$ of \eqref{eq:f'} is self-adjoint
 \end{itemize}
 are equivalent.\hfill $\qed$ \renewcommand{\qed}{}
\end{proof}

\appendix

\section{An aside on Hochschild splitting}\label{se:asidehoch}

For whatever intrinsic interest it may possess, as well as possible future reference, we record the following Hochschild-cohomology analogue of Lemma \ref{le:cansplit}. It could, in principle, have provided an alternative approach to the material in Section \ref{se:oaapp}, on spaces of operator-algebra morphisms.

\begin{Lemma}\label{le:hochsplit}
 Let $A$ be a finite-dimensional, semisimple Banach algebra and $E$ a Banach $A$-bimodule.

For every $n\in \bZ_{\ge 0}$, the map
 \begin{equation}\label{eq:ctobsurj-hoch}
 C^n(A,E)\xrightarrow{\quad} Z^{n+1}(A,E)
 \end{equation}
 is a split surjection. Furthermore, having fixed $A$ and $n$, there is a universal constant $K$ so that the splittings have norm $\le KC$ for a bound
 \begin{equation*}
 C\ge \text{norms of}\ A\otimes E\to E\qquad\text{and}\qquad E\otimes A\to E.
 \end{equation*}
\end{Lemma}
\begin{proof}
 The argument is parallel to the proof of Lemma \ref{le:cansplit}, proceeding via averaging.

 Because $A$ is semisimple over a perfect field, it is {\it separable} \cite[Section~10.2, Definition and Section~10.7, Corollary b]{pierce}, i.e., $A$ is projective in the category of $A$-bimodules. It follows that there is an element $e\in A\otimes A$ with
 \begin{equation*}
 ae=ea,\qquad\text{multiplication}\ (e)=1
 \end{equation*}
 (the {\it separability idempotent} of \cite[proof of Theorem 3]{cmz}, the {\it diagonal} of \cite[Section~7.2]{dael}, etc.). We can now use $e$ to ``average'' cocycles into cochains whose coboundaries they are, by analogy to Haar integration:

 If $a\in Z^{n+1}(A,E)$, then it will be the coboundary of $b\in C^n(A,E)$ defined by
 \begin{equation}\label{eq:splithcoc}
 \begin{tikzpicture}[auto,baseline=(current bounding box.center)]
 \path[anchor=base]
 (0,0) node (l) {$A^{\otimes n}$}
 +(3,.5) node (lu) {$A^{\otimes(n+2)}$}
 +(6,.5) node (ru) {$A\otimes E$}
 +(9,0) node (r) {$E$,}
 ;
 \draw[->] (l) to[bend left=6] node[pos=.5,auto] {$\scriptstyle e\otimes \id_{A^{\otimes n}}$} (lu);
 \draw[->] (lu) to[bend left=6] node[pos=.5,auto] {$\scriptstyle \id_A\otimes a$} (ru);
 \draw[->] (ru) to[bend left=6] node[pos=.5,auto,swap] {$\scriptstyle $} (r);
 \draw[->] (l) to[bend right=6] node[pos=.5,auto,swap] {$\scriptstyle b$} (r);
 \end{tikzpicture}
 \end{equation}
 with the upper right-hand map being the left-module action.
\end{proof}

\begin{Remarks}\quad
 \begin{enumerate}[(1)]\itemsep=0pt
 \item For the averaging procedure employed in the proof see also \cite[Section~3]{cg-average}, which deals with {\it virtual} diagonals.

 \item Finite groups are covered both by the present discussion and by that of Lemma \ref{le:cansplit}, and the two arguments coalesce: the separability idempotent of a finite-group algebra $A=\bC G$ is
 \begin{equation*}
 e = \frac 1{|G|}\sum_{g\in G}g\otimes g^{-1}
 \end{equation*}
 \cite[Exercise 2.2.1]{run-lec}, so \eqref{eq:splithcoc} specializes back to \eqref{eq:splitgpcoc}.

 \item For a slightly different perspective on the splitting, recall the {\it bar complex}
 \begin{equation}\label{eq:barcompl}
 \cdots
 \xrightarrow{\quad \partial_2\quad}
 A^{\otimes 3}
 \xrightarrow{\quad \partial_1\quad}
 A^{\otimes 2}
 \xrightarrow{\quad \partial_0\quad}
 A\to 0
 \end{equation}
 of \cite[Definition 4.2.1]{ginz-ncg}, providing a free resolution of $A$ (in the category of $A$-bimodules). The complex \eqref{eq:hochcoc} can then be obtained by applying the functor ${}_A\Hom_{A}(-,E)$ (morphisms of $A$-bimodules) to the bar complex, after deleting the last term $A$ and identifying
 \begin{equation}\label{eq:cnavatars}
 C^n(A,E) = \Hom\big(A^{\otimes n},E\big)\cong {}_A\Hom_{A}\big(A^{\otimes (n+2)},E\big).
 \end{equation}
 Now, because $A$ is separable, \eqref{eq:barcompl} is {\it contractible} \cite[Definition following Theorem 6.14]{rot}: appropriately-chosen $A$-bimodule maps
 \begin{equation*}
 A^{\otimes(n+2)}\xrightarrow{\quad h_n\quad} A^{\otimes(n+3)},\qquad n\ge -1
 \end{equation*}
 (a {\it contracting homotopy}, per loc.cit.) provide decompositions
 \begin{equation*}
 A^{\otimes(n+2)} \cong \left(\im\partial_{n+1} = \ker\partial_n\right)\oplus \im\partial_n.
 \end{equation*}
 The corresponding decomposition of $C^n(A,E)$ has the kernel of \eqref{eq:ctobsurj-hoch} as one of its summands (taking the identification \eqref{eq:cnavatars} for granted), hence the splitting. As for the universal-bound claim, note that the splitting depends only on the choice of contracting homotopy.
 \end{enumerate}
\end{Remarks}

\subsection*{Acknowledgements}

This work is partially supported by NSF grant DMS-2001128.
We are grateful for valuable pointers to the literature from Harl-Hermann Neeb, as well as the anonymous referees for insightful comments and suggestions.

\pdfbookmark[1]{References}{ref}
\LastPageEnding

\end{document}